\documentclass{amsart}
\usepackage{amsmath, amsfonts, latexsym, array, amssymb, amscd }
\usepackage[all]{xypic}
\usepackage[all]{xy}
\usepackage{graphicx}
\usepackage{color}
\usepackage{enumitem}
\usepackage{epstopdf}
\usepackage{hyperref}
\usepackage{mathtools}
\usepackage{mathrsfs}
\usepackage{parskip}
\usepackage{accents}

\numberwithin{equation}{section}
\numberwithin{figure}{section}

\newtheorem{theorem}{Theorem}[section]
\newtheorem*{theorem*}{Theorem}
\newtheorem{lemma}[theorem]{Lemma}

\newtheorem{proposition}[theorem]{Proposition}
\newtheorem*{proposition*}{Proposition}

\theoremstyle{definition}
\newtheorem{definition}[theorem]{Definition}

\usepackage{indentfirst}
\usepackage{ulem}

\begin{document}
\title[Thermodynamic formalism of the Katok map] {Unique equilibrium states, large deviations and Lyapunov spectra for the Katok map}
\author{Tianyu Wang}
\address{Department of Mathematics, The Ohio State University, Columbus, OH 43210, \emph{E-mail address:} \tt{wang.7828@buckeyemail.osu.edu}}

%\subjclass[2010]{37D35, 37D40, 37C40. 37D25}
\date{\today}
\thanks{This work is partially supported by NSF grant DMS-$1461163$.}
%\keywords{Equilibrium states, geodesic flow, Kolmogorov property}
%\commby{}

\begin{abstract}
We study the thermodynamic formalism of a $C^{\infty}$ non-uniformly hyperbolic diffeomorphism on the 2-torus, known as the Katok map. We prove for a H\"older continuous potential with one additional condition, or geometric t-potential $\varphi_t$ with $t<1$, the equilibrium state exists and is unique. We derive the level-2 large deviation principle for the equilibrium state of $\varphi_t$. We study the multifractal spectra of the Katok map for the entropy and dimension of level sets of Lyapunov exponents. 
\end{abstract}

\maketitle
\setcounter{tocdepth}{1}

\section{Introduction}
The Katok map is a $C^\infty$ non-uniformly hyperbolic toral automorphism in dimension 2, generated by a slow-down of the trajectories of a uniformly hyperbolic toral automorphism in a small neighborhood near the fixed point. So far the existence and uniqueness of equilibrium states for uniformly hyperbolic diffeomorphisms with sufficiently regular potentials are well-studied in \cite{Bow75-1}. Meanwhile, researchers have been able to derive the statistical properties for the equilibrium state via symbolic dynamics, including Bernoulli property, exponential decay of correlations and the Central Limit Theorem, see \cite{PP90}, \cite{Rue69}. 

Nevertheless, the thermodynamic formalism of non-uniformly hyperbolic systems is far away from being complete. In the case of the Katok diffeomorphism, non-uniform hyperbolicity is generated by the existence of a neutral fixed point. Its thermodynamic formalism has features in common with the model example of the one-dimensional Manneville-Pomeau map, admitting a neutral fixed point at zero. In \cite{PSZ16}, Pesin, Senti and Zhang studied Katok map as Young's diffeomorphism using countable Markov diagram. In \cite{SZ18}, Shahidi and Zelerowicz studied the Bernoulli properties and decay of correlations of the equilibrium state of the Katok map for locally H\"older potentials. This technique has been applied to other non-uniformly hyperbolic cases, see for example \cite{ST16}. 

In this paper, we study the Katok map using the orbit decomposition approach. The technique is first introduced in \cite{CT16}. The spirit is to generalize the dynamical properties for the map and regularity conditions for potential functions from \cite{Bow75-1} and make them hold on an ``essential collection of orbit segments'' which dominates in topological pressure and presents ``enough uniformly hyperbolic behavior''. This technique has been applied to other non-uniformly hyperbolic cases, see \cite{CFT17-1},\cite{CFT17-2} for DA (derived from Anosov) homeomorphisms, and \cite{BCFT18} for flows. We will compare our approach to that of \cite{PSZ16} after we state our results and explain the details in $\mathsection 7$.

One crucial fact about the Katok map is that it admits an equilibrium state for any continuous potential as the map is expansive. In fact, the Katok map is topologically conjugate to a linear torus automorphism via a homeomorphism, therefore has the specification property. By \cite{Bow75-1}, we know the Katok map has a unique measure of maximal entropy. However, since the conjugacy homeomorphism is neither differentiable nor H\"older, the thermodynamic formalism of Katok map is non-trivial. When the potential functions are geometric t-potentials, the Katok map will go through phase transition just like what happens to Manneville-Pomeau map.  We will prove for any $t<1$, there exists an orbit decomposition such that $t\varphi^{geo}$ has the required regularity on a collection of orbit segments that dominates in pressure. Applying Theorem A in \cite{CT16}, we are able to conclude the uniqueness of equilibrium states for all such $t\varphi^{geo}$. A similar result is obtained for H\"older potentials with the pressure gap $P(\delta_0)<P(\varphi)$, where $\delta_0$ is the Dirac measure at the origin.

Before we state the theorems, we make brief remarks on the notations. In the definition of the Katok map (see \cite{Kat79} and also $\mathsection 3$ for its properties) we have two parameters $r_0$ and $\alpha$. Roughly speaking, $r_0$ is the radius of the perturbed region and $\alpha$ describes the exponential slow-down rate. We also write $\varphi_t=t\varphi^{geo}$ with $\varphi^{geo}=-\log\vert D\widetilde{G}\vert _{E^u(x)}\vert$ being the geometric potential, where $E^u(x)$ is the unstable distribution of $D\widetilde{G}$ at $x$ and $\widetilde{G}$ is the Katok map.

\begin{theorem}
Given the Katok map $\widetilde{G}$ whose $\alpha$ and $r_0$ are sufficiently small, if $\varphi \in C(\mathbb{T}^2)$ is H\"older continuous and $\varphi(\underaccent{\bar}{0})<P(\varphi)$, where $\underaccent{\bar}{0}$ is the origin, then there is a unique equilibrium state for $\varphi$.
\end{theorem}

\begin{theorem}
Given the Katok map $\widetilde{G}$ whose $\alpha$ and $r_0$ are chosen sufficiently small, $\varphi_t$ has unique equilibrium state for $t\in (-\infty,1)$.
\end{theorem}

In Theorems 1.1 and 1.2, we want $\alpha$ and $r_0$ to be small enough so that the desired dynamical properties, i.e. specification, regularity for potential, etc. will hold for the good collection of orbit segments. For the details on how small the range is, see the end of $\mathsection 3$.

One benefit that \cite{Bow75-1} and \cite{CT16} bring us is to construct the unique equilibrium state as a Gibbs measure. In \cite{Bow75-1}, the lower Gibbs property is essential in ruling out the mutually singular equilibrium states. This approach is generalized in \cite{CT16}, in which Climenhaga and Thompson derive the lower Gibbs property of equilibrium state for the ``essential collection of orbit segments'' which dominates in pressure as well as an uniform upper Gibbs property for all orbit segments. In this paper, we are able to deduce a non-uniform version of upper and lower Gibbs property for all orbit segments at all scales. Based on this property and the entropy density, we are able to deduce the large deviation principle for the equilibrium state of Katok map for $t\varphi^{geo}$. In general, the large deviation principle describes the exponential rate of convergence of time average to the space average with respect to a given measure. The following theorem is proved in $\mathsection 8$.

\begin{theorem}
The unique equilibrium state derived in Theorem 1.1 and 1.2 has the level-2 large deviations principle.
\end{theorem}

The uniqueness result also helps us to study the multifractal spectra of level sets of Lyapunov exponents by estimating the dimension from below and giving the exact entropy. In $\mathsection 9$, we prove the following theorem:

\begin{theorem}
Let $\mathscr{P}(t):=P(t\varphi^{geo})$, $\alpha_1:=\lim_{t\rightarrow -\infty}D^{+}\mathscr{P}(t)$ and also $\alpha_2:=D^{-}\mathscr{P}(1)$. For all $\alpha\in (\alpha_1,0]$, $L(-\alpha)$ is non-empty. Moreover, we have its entropy to satisfy $h(L(-\alpha))=\mathscr{E}(\alpha)$, where $\mathscr{E}(\alpha)$ is the Legendre transform of $\mathscr{P}$ at $\alpha$ (see $\mathsection 9.1$ for the definition). When $\alpha \in (\alpha_1,0)$, the Hausdorff dimension of $L(-\alpha)$ satisfies $d_H(L(-\alpha))\geq \frac{-2\mathscr{E}(\alpha)}{\alpha}$ In particular, when $\alpha\in [\alpha_2,0)$, we have $d_H(-\alpha)=2$.
\end{theorem}
Here, $L(-\alpha)$ is the set of Lyapunov-regular points whose positive forward and backward Lyapunov exponent are both $-\alpha$ with $h(L(-\alpha))$ and $d_H(L(-\alpha))$ being its topological entropy and Hausdorff dimension. We notice that due to the existence of neutral fixed point, the pressure function $\mathscr{P}(t)$ goes through a phase transition at $t=1$, in particular $\alpha_2<0$. See $\mathsection 9$ for the definition of $\mathscr{E}(\alpha)$ and other details.

We briefly compare the above results to those in \cite{PSZ16}. They obtain the results of Theorem 1.2 when $t\in (t(\alpha,r_0),1)$ with $t(\alpha,r_0)\rightarrow -\infty$ when $\alpha,r_0\rightarrow 0$. Their question on if the range of t can be extended to $-\infty$ for a fixed Katok map is answered here, as the orbit decomposition approach will allow us to take $t(\alpha,r_0)=-\infty$ for fixed $\alpha,r_0$, which is the optimal uniqueness result for equilibrium states. Besides that, the large deviations and multifractal results are well suited to the specification approach and uniqueness results. On the other hand, \cite{PSZ16} emphasizes the statistical properties of the equilibrium state by the nature of the Tower construction. We refer the reader to $\mathsection 7$ for more technical details of the comparison.

The structure of the paper is as follows. In $\mathsection 2$, we introduce the orbit decomposition technique that we apply throughout the paper. In $\mathsection 3$, we briefly introduce the Katok map and deduce some relevant properties that will be used in constructing the orbit decomposition. In $\mathsection 4$, we establish the decomposition. In $\mathsection 5$, we prove that the essential collection in the decomposition dominates in pressure under certain conditions. In $\mathsection 6$, we prove the Bowen property for H\"older continuous potential functions and geometric t-potentials. In $\mathsection 7$, we conclude our Theorem 1.1 and 1.2 as our main results on uniqueness of equilibrium states. In $\mathsection 8$, we deduce the large deviation principle for the equilibrium states in Theorem 1.1 and 1.2 and thus deduce Theorem 1.3. In $\mathsection 9$, we study the multifractal spectra of the Katok map in terms of topological entropy and Hausdorff dimension and prove Theorem 1.4.

\section{Main technique}

We state the preliminary definitions needed for the technique and introduce how to apply the technique to deduce the desired thermodynamic formalism.

\subsection{Pressure}

Let $X$ be a compact metric space and $f:X\rightarrow X$ be a continuous map of finite topological entropy. Given a continuous real-valued function $\varphi$ on $X$ and call it potential(function). Denote the space of all $f$-invariant Borel probability measure on $X$ by $\mathcal{M}(f)$ and $\mathcal{M}_{e}(f)\subset \mathcal{M}(f)$ the ergodic ones. 

We write
$$S_n(\varphi)=S_n^{f}(\varphi)=\sum_{k=0}^{n-1}\varphi(f^kx).$$
Given $n\in \mathbb{N}$ and $x,y\in X$, we define 
$$d_n(x,y)=\max_{0\leq k\leq n-1}{d(f^k(x),f^k(y))}.$$
The Bowen ball of order $n$ at center $x$ with radius $\epsilon$ is defined as
$$B_n(x,\epsilon)=\{y\in X:d_n(x,y)<\epsilon\}.$$
We need to separate points using Bowen balls. Suppose $Y \subset X$ and $\delta>0$. We say $E \subset Y$ is a ($\delta,n$)-separated set if $d_n(x,y)\geq \delta$ for all $x\neq y$, $x,y\in E$. Write
$$\Lambda_n^{sep}(Y,\varphi,\delta;f)=\sup\left\{\sum_{x\in E}e^{S_n\varphi(x)} : E\subset Y \text{ is an } (\delta,n) \text{-separated set}\right\}.$$
The pressure of $\varphi$ on $Y$ is defined as
$$P(Y,\varphi;f)=\lim_{\delta \rightarrow 0}\limsup_{n\rightarrow \infty}\frac{1}{n}\log\Lambda_n^{sep}(Y,\varphi,\delta;f).$$
In particular, when $Y=X$, we write $P(X,\varphi;f)$ as $P(\varphi)$.

More generally, sometimes we must consider the pressure of a collection of orbit segments. As defined in \cite{CT16}, we interpret $\mathscr{D}\subset X \times \mathbb{N}$ as a collection of finite orbit segments and write $\mathscr{D}_n=\{ x\in X : (x,n)\in \mathscr{D} \}$. Consider the partition sum
$$\Lambda_n^{sep}(\mathscr{D},\varphi,\delta;f)=\sup\left\{\sum_{x\in E}e^{S_n\varphi(x)} : E\subset \mathscr{D}_n \text{ and is an } (\delta,n) \text{-separated set}\right\}$$
which enables us to define $P(\mathscr{D},\varphi;f)$ in the same way.

The variational principle from \cite{Wal82} says that:

$$P(\varphi)=\sup_{\mu \in \mathcal{M}(f)}\left\{h_{\mu}(f)+\int{\varphi}d\mu\right\}=
\sup_{\mu \in \mathcal{M}_{e}(f)}\left\{h_{\mu}(f)+\int{\varphi}d\mu\right\}.$$
A measure achieving the supremum is called equilibrium state. One of the main topics in this paper is to study the existence and uniqueness of this object.

Later in the estimate on pressure gap, we have to consider a following variation of the definition of pressure, which first appears in \cite{CT16}. Given a fixed scale $\epsilon>0$, we define
$$\Phi_{\epsilon}(x,n):= \sup_{y\in B_n(x,\epsilon)}\sum_{k=0}^{n-1}\varphi(f^ky).$$
From the above definition we see immediately that $\Phi_{0}(x,n)=\sum_{k=0}^{n-1}\varphi(f^kx)$.

For $\mathscr{D}\subset X \times \mathbb{N}$, we write
$$\Lambda_n^{sep}(\mathscr{D},\varphi,\delta,\epsilon;f)=\sup\left\{\sum_{x\in E}e^{\Phi_{\epsilon}(x,n)} : E\subset \mathscr{D}_n \text{ and is an } (\delta,n) \text{-separated set}\right\}.$$
The pressure of $\varphi$ on $\mathscr{D}$ at scale $\delta, \epsilon$ is given by
$$P(\mathscr{D},\varphi,\delta,\epsilon;f)=\limsup_{n\rightarrow \infty}\frac{1}{n}\log\Lambda_n^{sep}(\mathscr{D},\varphi,\delta,\epsilon;f).$$
Again, when $\mathscr{D}$ is the entire $X \times \mathbb{N}$, we simply write $P(\varphi,\delta,\epsilon)$.

\subsection{Specification, expansivity and regularity}

\subsubsection{Specification}

Specification describes the property that different Bowen balls can be connected by an orbit segment with uniform gap.

\begin{definition}
A collection of orbit segments $\mathscr{D}\subset X \times \mathbb{N}$ has specification at scale $\epsilon$ if there exists $\tau=\tau(\epsilon)\in \mathbb{N}$ such that for every $\{(x_j,n_j): 1\leq j \leq k\}\subset \mathscr{D}$, there is a point $x$ in 
$$\bigcap_{j=1}^{k}f^{-m_{j-1}}B_{n_j}(x_j,\epsilon),$$
where $m_0=0$ and $m_j=m_{j-1}+n_j+\tau$ for $j\geq 1$.

\end{definition}

Sometimes we are only interested in connecting orbit segments that are long enough. In these situations, it's natural to come up with the following weak version of specification.

\begin{definition}
A collection of orbit segments $\mathscr{D}\subset X \times \mathbb{N}$ has tail specification at scale $\epsilon$ if there is some $N_0\in \mathbb{N}$ such that $\mathscr{D}_{\geq N_0}:=\{(x,n)\in \mathscr{D} | n\geq N_0\}$ has specification.
\end{definition}

\subsubsection{Expansivity}

\begin{definition}
We write the set of non-expansive points at scale $\epsilon$ as $\text{NE}(\epsilon):= \{ x\in X: \Gamma_{\epsilon}(x)\neq \{x\}\}$. The map $f$ is expansive at scale $\epsilon$ if $\text{NE}(\epsilon)=\emptyset$. A $f$-invariant Borel probability measure is called almost expansive at scale $\epsilon$ if $\mu(\text{NE}(\epsilon))=0$.
\end{definition}

To see whether the set of non-expansive points at some scale is negligible regarding pressure, we need the following quantity. This is known as the pressure of obstructions to expansivity in \cite{CT16},\cite{CFT17-1},\cite{CFT17-2}.
$$P_{\exp}^{\perp}(\varphi,\epsilon)
=\sup_{\mu\in \mathcal{M}_e(f)}\Big\{ h_{\mu}(f)+\int\varphi d\mu : \mu(\text{NE}(\epsilon))>0 \Big\}
$$
From the definition we notice that if $P_{\mu}(\varphi)> P_{\exp}^{\perp}(\varphi,\epsilon)$ and $\mu$ is $f$-invariant and ergodic, then $\mu$ is almost expansive at scale $\epsilon$.

\subsubsection{Regularity for potential}

The following regularity for the potential function is required in our case.

\begin{definition}
Given $\mathscr{D}\subset X \times \mathbb{N}$, we say a function $\varphi:X\rightarrow \mathbb{R}$ has Bowen property on $\mathscr{D}$ at scale $\epsilon$ if there exists a constant $K=K(\varphi,\mathscr{D},\epsilon)$ such that $|S_n\varphi(x)-S_n\varphi(y)|<K$ for any $(x,n)\in \mathscr{D}$ and $y\in B_n(x,\epsilon)$. A function $\varphi$ has Bowen property on $\mathscr{D}$ if it has Bowen property on $\mathscr{D}$ at some scale (therefore smaller scale as well).
\end{definition}

\subsection{Orbit decomposition technique}

Now we have all the ingredients that we need to deduce the uniqueness of equilibrium states. The following orbit decomposition construction, which is first completely introduced in \cite{CT16}, will be the main technique that we will apply throughout the paper.

For a compact metric space $X$ and $f:X\rightarrow X$ being at least $C^{1+\alpha}$ in our case, a decomposition for a pair $(X,f)$ consists of three collections $\mathscr{P},\mathscr{G},\mathscr{S}\subset X \times \mathbb{N}$ and three functions $p,g,s:X \times \mathbb{N}\rightarrow \mathbb{N}$ such that for every $(x,n)\in X \times \mathbb{N}$, the values $p=p(x,n)$, $g=g(x,n)$, $s=s(x,n)$ satisfy $n=p+g+s$ and

$$(x,p)\in \mathscr{P},\quad (f^p(x),g)\in \mathscr{G},\quad (f^{p+q}(x),s)\in \mathscr{S}.$$
Meanwhile, for each $M\in \mathbb{N}$, write $\mathscr{G}^M$ for the set of orbit segments $(x,n)$ such that $p\leq M$, $s\leq M$. Here $(x,0)$ is assumed to be contained in all of three collections. This basically means some elements in the decomposition can be empty. The following theorem (Theorem 5.6 in \cite{CT16}) is the main tool that we apply in this paper.

\begin{theorem}
Let $X,f,\varphi$ be as above. Suppose there is an $\epsilon>0$ such that $P_{\exp}^{\perp}(\varphi,100\epsilon)<P(\varphi)$ and $(X,f)$ admits a decomposition $(\mathscr{P},\mathscr{G},\mathscr{S})$ with the following properties:

\begin{enumerate}
\item For each $M\geq 0$, $\mathscr{G}^M$ has tail specification at scale $\epsilon$.
\item $\varphi$ has the Bowen property at scale $100\epsilon$ on $\mathscr{G}$.
\item $P(\mathscr{P}\cup {\mathscr{S}},\varphi,\epsilon,100\epsilon)<P(\varphi)$
\end{enumerate}
Then there is an unique equilibrium state for $\varphi$.
\end{theorem}
There is no specific meaning behind the constant $100\epsilon$, while we do require expansivity and regularity to be controlled in a much larger scale due to the multiple application of specification. In particular, all the estimates will be safe once regularity for potential holds at scale $100\epsilon$.

Here we remark that the transition time for $\mathscr{G}^M$ is dependent on the choice of $M$. Specification at all scales for $\mathscr{G}$ implies specification at all scales for $\mathscr{G}^M$ for any $M$ due to a simple argument in modulus of continuity (see \cite{CT16} for detail). For the Katok map we can obtain specification at any small scale due to its conjugacy to linear automorphism. Nevertheless, the conjugacy homeomorphism is not H\"older continuous, which makes the thermodynamic formalism of the Katok map different from the well-studied uniformly hyperbolic models.

We add a final remark on the term $P(\mathscr{P}\cup {\mathscr{S}},\varphi,\epsilon,100\epsilon)$, the two-scale pressure defined in $\mathsection 2.1$. In \cite{CFT17-1} where specification at all scale is not expected, the authors put a variation term in the pressure gap estimate. This variation term can be obtained by breaking down the two scale pressure. In fact, it is not hard to see that $P(\mathscr{D},\varphi,\epsilon,100\epsilon)=P(\mathscr{D},\varphi,\epsilon)$ when $\varphi$ has Bowen property on $\mathscr{D}$ at scale $100\epsilon$. In our case, though Bowen property does not hold on $\mathscr{P}\cup {\mathscr{S}}$, we will give an argument in $\mathsection 5$ using the local product structure to remove the $100\epsilon$ term.

\section{The Katok map and its properties}

We collect the materials for the Katok map that we will need to build the decomposition with desired properties. The Katok Map is a $C^\infty$ diffeomorphism  of $\mathbb{T}^2$ which preserves Lebesgue measure and is non-uniformly hyperbolic. Katok \cite{Kat79} originally constructed the map to verify the existence of $C^\infty$ area-preserving Bernoulli diffeomorphisms of $\mathbb{D}^2$ that are sufficiently flat near $\partial\mathbb{D}^2$.

Consider the automorphism of $\mathbb{T}^2$ given by $A$ = $\begin{bmatrix}
2 & 1\\
1 & 1\\
\end{bmatrix}$, which is locally the time-one map of generated by the local flow of the following differential system:
$$\frac{ds_1}{dt}=s_1\log\lambda , \quad \frac{ds_2}{dt}=-s_2\log\lambda.$$
where $(s_1,s_2)$ is the coordinate representation in the eigendirections of $A$ and $\lambda$ equals the greater eigenvalue of $A$. We slow down the trajectories of the flow in a neighborhood of origin as follows: Choose a number $0<\alpha<1$ and a function $\psi:[0,1]\rightarrow [0,1]$ satisfying:

(1) $\psi$ is $C^\infty$ everywhere except for the origin.

(2) $\psi (0)=0$ and $\psi(r_0)=1$ for some $0<r_0<1$ and $r_0$ is close to 0.

(3) $\psi '(x)\geq 0$ and is non-increasing.

(4) $\psi(u)=(u/r_0)^{\alpha}$ for $0\leq u \leq \frac{r_0}{2}$.

\noindent where $r_0$ is very small. Let $D_r=\{(s_1,s_2):s_1^2+s_2^2\leq r^2\}$ and $\lambda$ be the larger eigenvalue of $A$. We also define $r_1={r_0}\log\lambda$. Now the trajectories are slowed down in $D_{r_1}$ at the rate of $\psi$, which induces the following differential system:

$$\frac{ds_1}{dt}=s_1\psi(s_1^2+s_2^2)\log\lambda , \quad \frac{ds_2}{dt}=-s_2\psi(s_1^2+s_2^2)\log\lambda.$$

Denote the time-one map of the local flow generated by this differential system by $g$. From the choice of $r_1$ and the assumption that $r_0$ is small one could easily see that the domain of $g$ contains $D_{r_1}$. Moreover, $f_A$ and $g$ coincide in some neighborhood of $\partial{D_{r_1}}$. Therefore, the following map
$$G(x)=\begin{cases}
A(x) \hfill \qquad \text{if } x\in \mathbb{T}^2\setminus D_{r_1}\\
g(x) \hfill \qquad \text{if } x\in D_{r_1}
\end{cases}$$
defines a homeomorphism of 2-torus which is $C^\infty$ everywhere except for the origin. One can verify that $G(x)$ preserves probability measure $d\nu=\kappa_0^{-1}\kappa dm$, where $\kappa$ is defined by
$$\kappa(s_1,s_2):=\begin{cases}
(\psi(s_1^2+s_2^2))^{-1}  \qquad \text{if } (s_1,s_2)\in D_{r_0}\\
1 \hfill \qquad \text{elsewhere}
\end{cases}$$
and $\kappa_0$ is the normalizing constant.

Furthermore, $G$ is perturbed to an area-preserving $C^\infty$ diffeomorphism  via a coordinate change. Define $\phi$ in $D_{r_1}$ as

$$\phi(s_1,s_2)=\frac{1}{\sqrt{\kappa_0(s_1^2+s_2^2)}}(\int_0^{s_1^2+s_2^2}\frac{du}{\psi(u)})^{\frac{1}{2}}(s_1,s_2),$$
and set $\phi$ identity outside $D_{r_0}$. $\phi$ transfers the measure $\nu$ into area and the map $\widetilde{G}:=\phi \circ G \circ \phi^{-1}$ is thus area-preserving. Moreover, one can check $\widetilde{G}$ is a $C^\infty$ diffeomorphism on 2-torus. It is called the Katok map. 

We add a comment on the property of $\phi$. Observe that $\phi$ is conformal in the sense of being a scalar product of identity at each point. Moreover, by writing $\phi(s_1,s_2)$ as $\frac{1}{\sqrt{\kappa_0}}(\int_0^{r^2}\frac{du}{\psi(u)})^{\frac{1}{2}}(\frac{s_1}{\sqrt{s_1^2+s_2^2}},\frac{s_2}{\sqrt{s_1^2+s_2^2}})$ with $r^2:=s_1^2+s_2^2$ and differentiating in $r$, together with property (2) of $\psi$ and standard geometric argument, we conclude that there is a constant $C=C(\alpha,r_0)$ such that $\frac{d(\phi(s_1,s_2),\phi(s_1',s_2'))}{d((s_1,s_2),(s_1',s_2'))}\geq \frac{C}{\sqrt{\kappa_0}}$ for all $(s_1,s_2),(s_1',s_2')\in \mathbb{T}^2$ such that $(s_1,s_2)\neq (s_1',s_2')$. Since $\phi$ is invertible, respectively we have 
\begin{equation}
    \begin{aligned}
    \frac{d(\phi^{-1}(s_1,s_2),\phi^{-1}(s_1',s_2'))}{d((s_1,s_2),(s_1',s_2'))} \leq \frac{\sqrt{\kappa_0}}{C}.
    \end{aligned}
\end{equation}
This property will be useful when we deduce the regularity of geometric potential of $\widetilde{G}$ from the regularity of geometric potential of $G$ in $\mathsection 7$.

We also remark on the connections between $G$ and $\widetilde{G}$. Since $\widetilde{G}$ is conjugate to $G$ via a homeomorphism that is $C^{\infty}$ everywhere except at the origin, the dynamical properties of $G$ are inherited by $\widetilde{G}$. The only place where the properties of $G$ and $\widetilde{G}$ need to be distinguished is the regularity of $\varphi^{geo}$ and $\varphi_G^{geo}$, referring to the geometric potential of $\widetilde{G}$ and $G$ respectively. Essentially these are two different potentials, so we want to analyze them separately. The idea will be to first prove the regularity of $\varphi_G^{geo}$, then use the property of $\phi$ and the conjugacy between $G$ and $\widetilde{G}$ to obtain the one for $\varphi^{geo}$.

\begin{proposition}
Here we have some useful properties of the Katok map \cite{Kat79}:
\begin{enumerate}
\item 
The Katok map is topologically conjugate to $f_A$ via a homeomorphism $h$, i.e. $\widetilde{G}=h \circ f_A \circ h^{-1}$. In fact, it is in the $C^0$ closure of Anosov diffeomorphisms, which means it is a $C^0$ limit of a sequence of Anosov diffeomorphisms.
\item
It admits two transverse invariant continuous stable and unstable distributions $E^s(x)$ and $E^u(x)$ that integrate to continuous, uniformly transverse and invariant foliations $W^s(x)$ and $W^u(x)$ with smooth leaves. Moreover, they are the image of stable and unstable eigendirections of $f_A$ under $h$.
\item
Almost every $x$ with respect to area $m$ has two non-zero Lyapunov exponents, one positive in the direction of $E^u(x)$ and the other negative in the direction of $E^s(x)$. The only ergodic measure with zero Lyapunov exponents is $\delta_0$, the point measure at the origin.
\item
It is ergodic with respect to $m$.

\end{enumerate}
\end{proposition}

In Proposition 3.1, properties (1) and (2) hold for $G$ with $h$ replaced by $\psi^{-1}\circ h$ and properties (3) and (4) hold for $G$ with respect to $\nu$. To get prepared for building the decomposition, let us first prove some propositions that will help and lead to the construction. The Ma\~n\'e and Bonatti-Viana's versions can be found in \cite{CFT17-1},\cite{CFT17-2}.

\begin{definition}
The leaves $W^s$ and $W^u$ are said to have local product structure with constant $\kappa$ at scale $\delta$, $\delta>0$, if the following holds: For any $x,y\in \mathbb{T}^2$, $d(x,y)<\delta$, there is a unique $z\in W^s_{\kappa \delta}(x)\cap W^u_{\kappa \delta}(y)$. 
\end{definition}

\begin{proposition}
When $\alpha$, $\epsilon>0$ are sufficiently small, the leaves $W^s$, $W^u$ of $G$ have local product structure at scale $500\lambda\epsilon$ with a constant only depending on $\alpha$.
\end{proposition}

Here we add a remark on the constant 500. There is no specific meaning behind the choice of this constant, while it has to be significantly large so that $500\lambda\epsilon$ will cover all the scales throughout the paper whose local product structure is needed (also $500\lambda\epsilon\ll 1$). We will see in the following sections that when $r_0$ and $\alpha$ are sufficiently small, the choice of $500\lambda\epsilon$ will work.

\begin{proof}
We want to show the leaves are contained respectively in $C_{\beta}(F^1,F^2)$ and $C_{\beta}(F^2,F^1)$, where $0<\beta<1$, $F^1,F^2$ are eigenspaces of $A$ with corresponding to $\lambda$ and $\lambda^{-1}$ and $C_{\beta}(F^1(x),F^2(x)):=\{x_1+x_2: x_1\in F^1(x), x_2\in F^2(x), \frac{\vert x_1 \vert}{\vert x_2 \vert}\leq \beta\}$. An application of Lemma 3.6 from \cite{CFT17-2}
will give local product structure with constant $\frac{1+\beta}{1-\beta}$. Moreover, we will prove $\beta$ is only depending on $\alpha$ (the exponent for the slow down function near origin) and converges to $0$ when $\alpha\rightarrow 0$.

We first prove the above cone argument, which is stated as the following lemma:
\begin{lemma}
There is a $0<\beta<1$ such that for all $x\in \mathbb{T}^2$, we have $$dG(C_{\beta}(F^1(x),F^2(x)))\subset C_{\beta}(F^1(G(x)),F^2(G(x))),$$
and
$$dG^{-1}(C_{\beta}(F^2(x),F^1(x)))\subset C_{\beta}(F^2(G^{-1}(x)),F^1(G^{-1}(x))),$$
where $F^1(x),F^2(x)$ are corresponding expanding and contracting eigenspaces in $T_x{\mathbb{T}^2}$. Moreover, $\beta$ only depends on $\alpha$ and $\beta\rightarrow 0$ when $\alpha \rightarrow 0$.
\end{lemma}

\begin{proof}
In \cite{Kat79} Katok proves the case where $\beta=1$. We follow the first step of the proof and then refine the result.

The differential system that generates the flow is
$$\frac{ds_1}{dt}=s_1\psi(s_1^2+s_2^2)\log\lambda , \quad \frac{ds_2}{dt}=-s_2\psi(s_1^2+s_2^2)\log\lambda.$$
As in Proposition 4.1 of \cite{Kat79}, consider the variation equation, which is the linear part of the above system, we get
$$\frac{d\xi_1}{dt}=\log\lambda(\xi_1(2s_1^2\psi'(s_1^2+s_2^2)+\psi(s_1^2+s_2^2))+2s_1s_2\xi_2\psi'(s_1^2+s_2^2)),$$
$$\frac{d\xi_2}{dt}=-\log\lambda(\xi_1s_1s_2\psi'(s_1^2+s_2^2)+\xi_2(2s_2^2\psi'(s_1^2+s_2^2)+\psi(s_1^2+s_2^2))).$$
By defining $\eta:=\frac{\xi_2}{\xi_1}$, we have
\begin{equation}
    \begin{aligned}
    \frac{d\eta}{dt}=-2\log\lambda(\eta(\psi(s_1^2+s_2^2)+(s_1^2+s_2^2)\psi'(s_1^2+s_2^2))+(\eta^2+1)s_1s_2\psi'(s_1^2+s_2^2)).
    \end{aligned}
\end{equation}

At first glance we should consider two cases where $s_1^2+s_2^2\leq \frac{r_0}{2}$ and $\frac{r_0}{2}<s_1^2+s_2^2 \leq r_0$. 
When $s_1^2+s_2^2\leq \frac{r_0}{2}$, we know what $\psi$ exactly is: Recall that $\psi(x)=(\frac{x}{r_0})^\alpha$. Then we have
$(s_1^2+s_2^2)\psi'(s_1^2+s_2^2)=\alpha\psi(s_1^2+s_2^2)$ for $0\leq s_1^2+s_2^2 \leq \frac{r_0}{2}$.

Otherwise, when $\frac{r_0}{2}\leq s_1^2+s_2^2 \leq r_0$, instead of an explicit equation between $\psi$ and $\psi'$, we have:
$$\frac{\psi'(s_1^2+s_2^2)}{\psi(s_1^2+s_2^2)}\leq \frac{\psi'(\frac{r_0}{2})}{\psi(\frac{r_0}{2})}=\frac{2\alpha}{r_0}\leq \frac{2\alpha}{s_1^2+s_2^2}.$$

It is then not hard to see that $\frac{\psi'(x)}{\psi(x)}\leq \frac{2\alpha}{x}$ for all $0<x\leq r_0$. Plugging this into (3.2) we have the following inequality: 
\begin{equation}
    \begin{aligned}
\frac{d\eta}{dt}\geq -2\log\lambda(\psi'(s_1^2+s_2^2)((s_1^2+s_2^2)(1+\frac{1}{2\alpha})\eta+s_1s_2(1+\eta^2))).
    \end{aligned}
\end{equation}

The case where $s_1s_2=0$ is easy to analyze using (3.1), as $\eta$ is decreasing when $\eta>0$ and increasing when $\eta<0$. We only analyze the case where $s_1,s_2>0$ because of symmetry. Observe from (3.2) that $\frac{d\eta}{dt}<0$ when $\eta \geq 0$, thus we only need to focus on $\eta<0$. By defining $k:=\frac{s_1s_2}{s_1^2+s_2^2}$ and doing some elementary calculation, We conclude that $\frac{d\eta}{dt}\geq 0$ when $\eta\in [\frac{-(2\alpha+1)-\sqrt{(2\alpha+1)^2-16k^2\alpha^2}}{4k\alpha},\frac{-(2\alpha+1)+\sqrt{(2\alpha+1)^2-16k^2\alpha^2}}{4k\alpha}]$. As $0<k\leq \frac{1}{2}$, the range of the slope of the invariant cone under all possible k values will be $\bigcap_{k\in (0,\frac{1}{2}]}[\frac{(2\alpha+1)-\sqrt{(2\alpha+1)^2-16k^2\alpha^2}}{4k\alpha},\frac{(2\alpha+1)+\sqrt{(2\alpha+1)^2-16k^2\alpha^2}}{4k\alpha}]$. Observe that $\frac{(2\alpha+1)-\sqrt{(2\alpha+1)^2-16k^2\alpha^2}}{4k\alpha}$ is monotonically increasing in $k$, so by plugging in $k=\frac{1}{2}$, we obtain an invariant cone with slope $\beta:=\frac{2\alpha}{2\alpha+1+\sqrt{4\alpha+1}}$.   
\end{proof}

Besides the above cone argument, we also need the following lemma on global structure on Euclidean space.

\begin{lemma}
Given $\beta\in (0,1)$ and $F^1,F^2\subset \mathbb{R}^d$ being orthogonal linear subspaces such that $F^1\cap F^2=\{0\}$, let $W^1,W^2$ be any foliations of $F^1\oplus F^2$ with $C^1$ leaves such that $T_x{W^1(x)}\subset C_{\beta}(F^1,F^2)$ and $T_x{W^2(x)}\subset C_{\beta}(F^2,F^1)$. Then for every $x,y\in F^1\oplus F^2$, $W^1(x)\cap W^2(y)$ consists of a single point. Moreover, 
\begin{equation}
\begin{aligned}
\max\{d_{W^1}(x,z),d_{W^2}(y,z)\}\leq \frac{1+\beta}{1-\beta}d(x,y).
\end{aligned}
\end{equation}

\end{lemma}

The proof is based on the elementary trigonometry and basic cone estimate. For a detailed proof of a more general version, see Lemma 3.6 in \cite{CFT17-1}.

With the help of Lemma 3.4 and 3.5, we are able to conclude local product structure for $G$ at $500\lambda\epsilon$, provided $\epsilon$, $\alpha$ and $r_0$ are all sufficiently small and $r_0\leq \epsilon$. We remark that the requirement of $\epsilon$, $\alpha$ being small is straightforward from the proof below, while the requirement of $r_0$ being small is needed to have $500\lambda\epsilon$ cover all the scales containing $r_0$ and $r_1$ throughout the paper, so that these scales will also possess local product structure with the same constant. This can also be visualized later in $\mathsection 4$ and $\mathsection 6$ when we choose the range used in the orbit decomposition for the regularity of potential.

We lift $W^s$ and $W^u$ to $\widetilde{W}^{s}$ and $\widetilde{W}^{u}$ in $\mathbb{R}^2$. Choose any $x,y\in \mathbb{T}^2$ such that $d(x,y)<500\lambda\epsilon$. From now on we use $\epsilon':=500\lambda\epsilon$ in this proof. We also use the notation $\gamma=\gamma(\beta):=\frac{1+\beta}{1-\beta}$ throughout the paper. Let $\widetilde{x},\widetilde{y}\in \mathbb{R}^2$ be lifts of $x,y$ such that $\widetilde{d}(\widetilde{x},\widetilde{y})<\epsilon'$. By Lemma 3.4 and Lemma 3.5 we know $\widetilde{W}^s(\widetilde{x})\cap \widetilde{W}^u(\widetilde{y})$ has a unique intersection $\widetilde{z}\in \mathbb{R}^2$. By projecting $\widetilde{z}$ back to $\mathbb{T}^2$ and (3.4), since $\beta$ and $\epsilon$ (thus $\epsilon'$) are chosen small so that the local leaf is not long enough to wrap around the torus, we have $z\in W^s_{\gamma\epsilon'}(x)\cap W^u_{\gamma\epsilon'}(y)$. Here $W^s_{\gamma\epsilon'}(x)$ means the local stable leaf of $x$ with radius $\gamma\epsilon'$.

Now it suffices to show $z$ is the only point in $W^s_{\gamma\epsilon'}(x)\cap W^u_{\gamma\epsilon'}(y)$. Suppose there is an other $z'\in \mathbb{T}^2$ also in $W^s_{\gamma\epsilon'}(x)\cap W^u_{\gamma\epsilon'}(y)$. Let $\gamma_1:[0,1]\rightarrow \mathbb{T}^2$ be any path that first connects $z$ and $z'$ via $W^s_{\gamma\epsilon'}(x)$ and then $z'$ and $z$ via $W^u_{\gamma\epsilon'}(y)$. Lift $\gamma_1$ to $\widetilde{\gamma_1}$ in $\mathbb{R}^2$, we notice $\widetilde{\gamma_1}(0)\neq \widetilde{\gamma_1}(1)$ since otherwise $\widetilde{W}^s(\widetilde{z})\cap \widetilde{W}^u(\widetilde{z})$ will not be unique. Observe $L(\gamma_1)\geq 1$ since $\gamma_1(0)=\gamma_1(1)$ while $\widetilde{\gamma_1}(0)\neq \widetilde{\gamma_1}(1)$. This contracts to the fact that $\epsilon'$ is small enough since the length of $\gamma_1$ is at most $2\gamma\epsilon'$, which is small.
\end{proof}

From now on we will assume $\alpha$ is fixed and so small such that $\beta$ is sufficiently small. This is possible by Lemma 3.4. As a result $\gamma$ will be very close to one and both $\lambda(1-\beta)$ and $\lambda(1+\beta)$ will be very close to $\lambda$, thus greater than one. We also fix $\epsilon$ to be sufficiently small such that Proposition 3.3 holds, as well as make $r_0$ small for the future use (as explained after stating Lemma 3.5). As a final comment, $r_0\leq \epsilon$ and the choice of $\epsilon$ is independent of the size of the gap $P(\varphi)-\varphi(\underaccent{\bar}{0})$.

By Proposition 3.1(1) and the fact that $f_A$ has specification at all scales, we have the following proposition.

\begin{proposition}
G has specification at all scales.
\end{proposition}

\subsection{Expansivity}

We know from Proposition 3.1(1) that $G$ is expansive. In this section we prove that $G$ is expansive at scale $100\epsilon$. 

Before giving the proof, we first prove a lemma which will be used very often throughout the paper.

\begin{lemma}
If $x,y\in \mathbb{T}^2$ and $y\in B_n(x,100\epsilon)$ for $\epsilon$ as above and $n\geq 1$, then we have a unique $z\in \mathbb{T}^2$ such that $G^i(z)\in W^s_{100\gamma\epsilon}(G^i(x))\cap W^u_{100\gamma\epsilon}(G^i(y))$ for all $0\leq i \leq n-1$.
\end{lemma}

\begin{proof}
Recall $\epsilon$ and $\beta$ are chosen small so that we have local product structure at $500\lambda\epsilon$. Fix any $x\in \mathbb{T}^2$ and $y\in B_n(x,100\epsilon)$. Since $d(G^i(x),G^i(y))\leq 100\epsilon$ for any $0\leq i \leq n-1$, by Proposition 3.3 and Lemma 3.5 we have $z_i\in \mathbb{T}^2$ such that $z_i=W^s_{100\gamma\epsilon}(G^i(x))\cap W^u_{100\gamma\epsilon}(G^i(y))$ for any $0\leq i \leq n-1$. Since $G(z_i)=W^s_{100\lambda(1+\beta)\gamma\epsilon}(G^{i+1}(x)) \cap W^u_{100\lambda(1+\beta)\gamma\epsilon}(G^{i+1}(x))$, by applying local product structure at scale $100\lambda(1+\beta)\gamma\epsilon$, we observe that $G(z_i)=z_{i+1}$, thus $G^i(z_0)=z_i$. It follows that $z_0$ is our desired $z$. 
\end{proof}

\begin{proposition}
$G$ is expansive at scale $100\epsilon$. In particular, $P_{\exp}^{\perp}(\varphi,100\epsilon)<P(\varphi)$.
\end{proposition}

\begin{proof}
Suppose there exists $x,y\in \mathbb{T}^2$ such that $d(G^k(x),G^k(y))<100\epsilon$ for any $k\in \mathbb{Z}$. By Lemma 3.7 to $B_n(x,100\epsilon)$ with each $n>0$, we have a $z\in \mathbb{T}^2$ such that $G^i(z)=W^s_{100\gamma\epsilon}(G^i(x))\cap W^u_{100\gamma\epsilon}(G^i(y))$ for all $i>0$.

For $i>0$, as $G^i(z)\in W^s_{100\gamma\epsilon}(G^i(x))$, we have $d(G^i(x),G^i(z))\leq 100\gamma\epsilon$. Therefore $d(G^i(y),G^i(z))\leq 100(1+\gamma)\epsilon$ for all $i>0$. From Lemma 3.7 in \cite{CFT17-1}, as $G^i(y)$ and $G^i(z)$ are always in the same local leaf of $W^u$, $d_u(G^i(y),G^i(z))\leq \gamma d(G^i(y),G^i(z)) \leq 100(1+\gamma)\gamma\epsilon$ for all $i>0$, which contradicts $z\in W^u_{100\gamma\epsilon}(y)$.
\end{proof}

\section{Construction of the decomposition}

Since the specification property holds globally for all the orbit segments at all scales, it suffices to choose  $\mathscr{G}$ in a way such that desired potentials have Bowen property. Meanwhile, $\mathscr{G}$ should be large enough so that pressure supported on $\mathscr{P} \cup \mathscr{S}$ is small. Consider the following set of orbit segments:
$$\mathscr{G}(r)=\{(x,n):\frac{1}{i}S_i\chi(x)\geq r \text{ and } \frac{1}{i}S_i\chi(G^{n-i}(x))\geq r \text{ for all }0\leq i \leq n \}$$
where $\chi$ is the characteristic function for $\mathbb{T}^2\setminus D_{100\gamma\epsilon+r_1}$ and $r$ is defined on $(0,1]$. In practice, we only consider the case where $r$ is small. The choice of constants in $\chi$ is to make sure that orbit segments that start and end far away from origin and spend enough time outside the perturbed area would show high regularity for the chosen family of potential function.

Respectively we choose 
$$\mathscr{P}(r)=\mathscr{S}(r)=\{(x,n)\in \mathbb{T}^2 \times \mathbb{N}: \frac{1}{n}S_n\chi(x)<r\}$$

The case where $n=0$ shall not cause ambiguity as we have $\mathbb{T}^2\times \{0\}$ to be contained in all of three collections. We will see later in $\mathsection 5$ and $\mathsection 6$ that the appropriate choice of $r$ will make Theorem 2.5 applicable to $(\mathscr{P}(r), \mathscr{G}(r), \mathscr{S}(r))$. Before moving forward to the verification of those properties, we must prove they actually form an orbit decomposition.

\begin{proposition}
For every $0<r \leq 1$, the collections $(\mathscr{P}(r), \mathscr{G}(r), \mathscr{S}(r))$ form an orbit decomposition for $G$.
\end{proposition}

\begin{proof}
For $(x,n)\in \mathbb{T}^2\times \mathbb{N}$, consider the largest integer $0\leq i \leq n$ such that $S_i\chi(x)<ir$ and the largest integer $0\leq k \leq n-i$ such that $S_k\chi(G^{n-k}(x))<kr$. If $S_j\chi(x)\geq jr$ for all $0\leq j \leq n$, we take $i=0$ (the case for $k$ is similar). By the definition of $i$ and $k$ we have $\frac{1}{l}S_l\chi(G^i(x))\geq r$ for $0\leq l \leq n-i$ and $\frac{1}{m}S_m\chi(G^{n-k-m}(x))\geq r$ for $0\leq m \leq n-k$. Therefore, we have
$$(x,i)\in \mathscr{P}(r), \quad (G^ix,n-i-k)\in \mathscr{G}(r), \quad (G^{n-k}x,k)\in \mathscr{S}(r)$$
which concludes the proof.
\end{proof}

\section{Pressure Gap}

We want to prove that given $\varphi(\underaccent{\bar}{0})<P(\varphi)$, we can find $r'>0$ sufficiently small so that $P(\mathscr{P}(r'),\varphi,\epsilon,100\epsilon)<P(\varphi)$. We first show that there is an $r'$ that $P(\mathscr{P}(r'),\varphi)<P(\varphi)$. Then we get $P(\mathscr{P}(r'),\varphi,\epsilon)<P(\varphi)$ automatically, as $P(\mathscr{P}(r'),\varphi,\epsilon)<P(\mathscr{P}(r),\varphi)$. Finally we show $P(\mathscr{P}(r'),\varphi,\epsilon)=P(\mathscr{P}(r'),\varphi,\epsilon,100\epsilon)$ in our case. This yields the third condition in Theorem 2.5, with $\mathscr{P}$ being chosen as $\mathscr{P}(r')$.

\subsection{General estimates}

We start with a general estimate for pressure on set of orbit segments. Under the same setting and given $\mathscr{D}\subset X\times \mathbb{N}$, for $(x,n)\in \mathscr{D}$, we define the empirical measure $\delta_{x,n}$ by
$$\delta_{x,n}:=\frac{1}{n}\sum_{i=0}^{n-1}\delta_{G^i(x)}.$$
For each $n\in \mathbb{N}$ we consider the following convex hull of $\delta_{x,n}$ for $(x,n)\in \mathscr{D}$:
$$\mathscr{M}_n(\mathscr{D}):=\left\{\sum_{i=1}^{k}a_i\delta_{x_i,n} : a_i\geq 0, \sum{a_i}=1,x_i\in \mathscr{D}_n\right\}.$$

Denote the weak* limit points of $\mathscr{M}_n(\mathscr{D})$ when $n\rightarrow \infty$ by $\mathscr{M}^{*}(\mathscr{D})$, we observe that $\mathscr{M}^{*}(\mathscr{D})$ is non-empty when $P(\mathscr{D},\varphi)>-\infty$ and $\mathscr{M}^{*}(\mathscr{D})\subset \mathscr{M}(X)$. 

Following the standard proof of variational principle for pressure in \cite{Wal82} (or see Proposition 5.1 in \cite{BCFT18}), we get

\begin{proposition}
$P(\mathscr{D},\varphi)\leq \sup_{\mu\in \mathscr{M}^{*}(\mathscr{D})}P_{\mu}(\varphi)$.
\end{proposition}

\subsection{Pressure gap estimate}
We notice that the measures in $\mathscr{M}^{*}(\mathscr{P}(r))$ are the weak* limits of measures in $\mathscr{M}_n(\mathscr{P}(r))$ when $n\rightarrow \infty$. For $\mu_n\in \mathscr{M}_n(\mathscr{P}(r))$, we observe that $\int{\chi}d\mu_n < r$ by definition of $\mathscr{P}(r)$. For each $0<r \leq 1$, write $\mathscr{M}_{\chi}(r)$ to be the set of $G$-invariant Borel probability measures $\mu $ such that $\int{\chi}d\mu\leq r$. Observe that $\mathscr{M}_n(\mathscr{P}(r))\subset \mathscr{M}_{\chi}(r)$ for any $n\in \mathbb{N}$. The following lemma says that this inclusion holds true in the limit case.

\begin{lemma}
 $\mathscr{M}^{*}(\mathscr{P}(r))\subset \mathscr{M}_{\chi}(r).$
\end{lemma}

In fact, Lemma 5.2 follows easily from the following lemma concerning the weak*-compactness of the set $\mathscr{M}_{\chi}(r)$, for which we will give a proof.

\begin{lemma}
$\mathscr{M}_{\chi}(r)$ is weak*-compact for all $0<r \leq 1$.
\end{lemma}

\begin{proof}
Without loss of generality we assume $\mu$ is the weak* limit of $\{\mu_{n_k}\}_{k\geq 1}$, where $\mu_{n_k}\in \mathscr{M}_{\chi}(r)$. We want to show that $\int{\chi}d\mu \leq r$.
Recall $\chi$ is the characteristic function for $\mathbb{T}^2\setminus D_{100\gamma\epsilon+r_1}$, thus lower-semi continuous, as we define $D_r$ to be the closed balls. Then $\int{\chi d_{\mu}}\leq \liminf_{k\rightarrow \infty}\int{\chi d_{\mu_{n_k}}}\leq r$ by remarks preceding Theorem 6.5 in \cite{Wal82}.
\end{proof}

We first observe that $\mathscr{M}_{\chi}(r)$ is non-decreasing in $r$ and $\mathscr{M}_{\chi}(0)=\bigcap_{r>0}\mathscr{M}_{\chi}(r).$ For $\mu \in \mathscr{M}_{\chi}(0)$, $\mu(\mathbb{T}^2\setminus{D_{100\gamma\epsilon+r_1}})=0$. However, we have $\bigcup_{k=-\infty}^{+\infty}G^{k}(\mathbb{T}^2\setminus{D_{100\gamma\epsilon+r_1}})=\mathbb{T}^2\setminus \{\underaccent{\bar}{0}\}$. By invariance of $\mu$, we conclude that $\mu=\delta_0$, the Dirac measure at origin, thus $\mathscr{M}_{\chi}(0)=\delta_0$, and $P_{\delta_0}(\varphi)=\varphi(\underaccent{\bar}{0})$.

Meanwhile, from Proposition 3.8, we know $G$ is expansive, so the entropy function $\mu\rightarrow h_{\mu}(\varphi)$ is upper semi-continuous, so is the pressure function $\mu\rightarrow P_{\mu}(\varphi)$. Therefore, for any small $\epsilon'>0$, there is an open neighborhood $U$ of $\delta_0$ in the weak* topology of $\mathscr{M}(X)$ such that for any $\mu\in U$, we have $P_{\mu}(\varphi)<P_{\delta_0}(\varphi)+\epsilon'=\varphi(\underaccent{\bar}{0})+\epsilon'$. By Lemma 5.3, there exists some $r'>0$ such that $\mathscr{M}_{\chi}(r')\subset U$. Since $\varphi(\underaccent{\bar}{0})<P(\varphi)$, by taking $0<\epsilon'<P(\varphi)-\varphi(\underaccent{\bar}{0})$, we obtain respective $r'>0$ such that $\sup_{\mu\in \mathscr{M}_{\chi}(r')}P_{\mu}(\varphi)\leq \varphi(\underaccent{\bar}{0})+\epsilon'<P(\varphi)$. This together with Proposition 5.1 and Lemma 5.2 show that $P(\mathscr{P}(r'),\varphi)<P(\varphi)$ for the $r'$ in the proof.

\begin{proposition}
When $\varphi$ is a continuous potential function such that $\varphi(\underaccent{\bar}{0})<P(\varphi)$, there is some small $r'>0$ such that $P(\mathscr{P}(r'),\varphi)<P(\varphi)$.
\end{proposition}

\subsection{Two-scale estimate}

Now we want to show that $$P(\mathscr{P}(r'),\varphi,\epsilon)=P(\mathscr{P}(r'),\varphi,\epsilon,100\epsilon).$$

Recall that
\begin{equation}
\begin{aligned}
P(\mathscr{P}(r'),\varphi,\epsilon)=\limsup_{n\rightarrow \infty}\frac{1}{n}\log\Lambda_n^{sep}(\mathscr{P}(r'),\varphi,\epsilon;G), \\
P(\mathscr{P}(r'),\varphi,\epsilon,100\epsilon)=\limsup_{n\rightarrow \infty}\frac{1}{n}\log\Lambda_n^{sep}(\mathscr{P}(r'),\varphi,\epsilon,100\epsilon;G).
\end{aligned}
\end{equation}

We make the following definition of the variation term of $\varphi$ in degree $n$ at scale $100\epsilon$, which is used throughout this section and $\mathsection 8$.
\begin{definition}
$\zeta(n)=\zeta(n,\varphi,100\epsilon):=\sup_{x\in X,y\in B_n(x,100\epsilon)}|S_n\varphi(y)-S_n\varphi(x)|$.
\end{definition}

Observe that 
$$\Lambda_n^{sep}(\mathscr{P}(r'),\varphi,\epsilon;G) \leq \Lambda_n^{sep}(\mathscr{P}(r'),\varphi,\epsilon,100\epsilon;G) \leq \Lambda_n^{sep}(\mathscr{P}(r'),\varphi,\epsilon;G)e^{\zeta(n)}.$$

In order to eliminate the scale $100\epsilon$, we prove the following lemma:

\begin{lemma}
$\limsup_{n\rightarrow \infty}\frac{1}{n}\zeta(n)=0$.
\end{lemma}

We notice that the definition of $\zeta$ is not restricted to any of the collection of orbit segments. This will be particularly useful in $\mathsection 8$, where we try to obtain the uniform Gibbs property in a weak sense.
\begin{proof}
Recall that we have local product structure at $500\lambda\epsilon$. We know from Lemma 3.7 that for any $x\in \mathbb{T}^2$ and $y\in B_n(x,100\epsilon)$, there exists $z\in \mathbb{T}^2$ such that $G^i(z)=W^s_{100\gamma\epsilon}(G^i(x))\cap W^u_{100\gamma\epsilon}(G^i(y))$ for any $0\leq i \leq n-1$. We have 
\begin{equation}
    \begin{aligned}
    \zeta(n)
    &=\sup_{x\in \mathbb{T}^2,y\in B_n(x,100\epsilon)}|S_n\varphi(y)-S_n\varphi(x)|\\
    &\leq \sup_{x\in \mathbb{T}^2,y\in B_n(x,100\epsilon)}(|S_n\varphi(x)-S_n\varphi(z)|+|S_n\varphi(z)-S_n\varphi(y)|)\\
    &\leq \sup_{x\in \mathbb{T}^2,z\in W^s_{100\gamma\epsilon}(x)}|S_n\varphi(x)-S_n\varphi(z)|
    +\\
    &\sup_{y\in \mathbb{T}^2,G^{n-1}(z)\in W^u_{100\gamma\epsilon}(G^{n-1}(y))}|S_n\varphi(z)-S_n\varphi(y)|
    \end{aligned}
\end{equation}

To prove the lemma, it suffices to prove the following lemma:

\begin{lemma}
Define $\zeta^s(n):=\sup_{x\in \mathbb{T}^2,z\in W^s_{100\gamma\epsilon}(x)}|S_n\varphi(x)-S_n\varphi(z)|$. We have $\limsup_{n\rightarrow \infty}\frac{1}{n}\zeta^s(n)=0$. Similarly,  $\zeta^u(n):=\sup_{y\in \mathbb{T}^2,G^{n-1}(z)\in W^u_{100\gamma\epsilon}(G^{n-1}(y))}|S_n\varphi(z)-S_n\varphi(y)|$. As above, we have $\limsup_{n\rightarrow \infty}\frac{1}{n}\zeta^u(n)=0$. 
\end{lemma}

To prove the first part of Lemma 5.7, we define 
$$d^s_n(x):=\max\{d(G^{n-1}(x),G^{n-1}(z)),z\in W^s_{100\gamma\epsilon}(x),d_s(x,z)=100\gamma\epsilon\}$$
for each $n\geq 1$ and $x\in \mathbb{T}^2$. Here the maximum makes sense as we only have two possible choice in $z$ when $x$ is given. We notice that along local stable leaf, $\{d^s_n(x)\}_{n\geq 1}$ is a sequence of continuous functions that pointwise converges to $0$ and $d^s_n(x)\geq d^s_{n+1}(x)$. As $\mathbb{T}^2$ is compact, the convergence of $d^s_n(x)$ to $0$ is uniform.

We want to show for any small $\epsilon_0>0$, there's $N=N(\epsilon_0)\in \mathbb{N}$ large enough such that $\frac{1}{n}\zeta^s(n)<\epsilon_0$ for any $n>N$. $\varphi$ is continuous on $\mathbb{T}^2$, thus uniformly continuous. For fixed small $\epsilon_0>0$, there exists $\delta_0>0$ such that when $x,y\in \mathbb{T}^2$, $d(x,y)<\delta_0$, we have $|\varphi(x)-\varphi(y)|<\frac{\epsilon_0}{2}$. By uniform convergence of $d_n^s$, there exists $m_0\in \mathbb{T}^2$ such that $d_n^s(x)<\delta_0$ for any $n>m_0$. Therefore $\zeta^s(n)<2m_0\varphi_0+\frac{(n-m_0)\epsilon_0}{2}$, where $\varphi_0:=\sup_{x\in \mathbb{T}^2}\varphi(x)$. Now it is clear that we can choose some $N \in \mathbb{N}$ such that $\frac{1}{n}\zeta^s(n)<\epsilon_0$ for all $n>N$. By making $\epsilon_0$ go to $0$, we end the proof of Lemma 5.7.

To prove the second part, instead of $d_n^s(x)$, we define a function $d_n^u(x)$ by $d^u_n(x):=\max\{d(x,z),f^{n-1}(z)\in W^u_{100\gamma\epsilon}(G^{n-1}(x)),d_u(G^{n-1}(x),G^{n-1}(z))=100\gamma\epsilon\}$. We obtain $d^u_n(x)$ converges uniformly to $0$, prove for any small $\epsilon_0$ we can find some $M=M(\epsilon_0)\in \mathbb{N}$ such that $\frac{1}{n}\zeta^u(n)<\epsilon_0$ for all $n>M$.

By applying Lemma 5.7 to (5.2), we complete the proof of Lemma 5.6.
\end{proof}

From (5.1) and Lemma 5.6 we have
\begin{equation}
\begin{aligned}
P(\mathscr{P}(r'),\varphi,\epsilon,100\epsilon) 
&=\limsup_{n\rightarrow \infty}\frac{1}{n}\log\Lambda_n^{sep}(\mathscr{P}(r'),\varphi,\epsilon,100\epsilon;f) \\
&\leq \limsup_{n\rightarrow \infty}\frac{1}{n}\log\Lambda_n^{sep}(\mathscr{P}(r'),\varphi,\epsilon;f)+\limsup_{n\rightarrow \infty}\frac{1}{n}\zeta(n) \\
&=\limsup_{n\rightarrow \infty}\frac{1}{n}\log\Lambda_n^{sep}(\mathscr{P}(r'),\varphi,\epsilon;f) \\
&=P(\mathscr{P}(r),\varphi,\epsilon). 
\end{aligned}
\end{equation}
which is the desired result for pressure gap.

Finally we add a comment on the gap condition $\varphi(\underaccent{\bar}{0})<P(\varphi)$. As both the left and right sides of the inequality changes continuously in $\varphi$ in the $C^0$ topology, we know the set of continuous potentials satisfying this gap condition is $C^0$-open. In fact, it is not hard to show that it is also $C^0$-dense, using the ergodic measures are entropy dense in the space of invariant measures. Further results concerning how common the gap is could be interesting and we leave that to the reader to explore. 
\section{Regularity of potential functions}

From the previous section we obtain the desired pressure estimate on the bad orbit segments for continuous potential $\varphi$ with $\varphi(\underaccent{\bar}{0})<P(\varphi)$. In this section we will verify the regularity condition required by Theorem 2.5. We will focus on the family of H\"older continuous potentials and geometric t-potential $\varphi _t^G(x)=t\varphi^{geo}_G(x)=-t\log\vert DG\vert _{E^u(x)}\vert$. We first state a result about the uniform expansion/contraction along local leaves $W^u$/$W^s$ of orbit segments in $\mathscr{G}(r)$.

\begin{lemma}
For $(x,n)\in \mathscr{G}(r)$ and $y\in W^s_{100\gamma\epsilon}(x)$, we have $d_s(G^i(x),G^i(y))\leq (\lambda(1-\beta))^{-ir}d_s(x,y)$ for any $0\leq i \leq n-1$. Similarly, for $(x,n)\in \mathscr{G}(r)$ and $f^{n-1}(y)\in W^u_{100\gamma\epsilon}(f^{n-1}(x))$ and $0\leq j \leq n-1$, we have $d_u(G^j(x),G^j(y))\leq (\lambda(1-\beta))^{-(n-1-j)r}d_u(f^{n-1}(x),f^{n-1}(y))$.
\end{lemma}

\begin{proof}
For any point $z$ lying on $W^s_{100\gamma\epsilon}(x)$ between $x$ and $y$, when $\chi(G^i(x))=1$, since $d(G^i(x),G^i(z))\leq 100\gamma\epsilon$, $G^i(z)$ is outside the perturbed area, therefore $\Vert DG\vert _{E^s(z)}\Vert \leq (\lambda(1-\beta))^{-1}$. Therefore, we have $\vert DG^i\vert _{E^s(z)}\vert \leq (\lambda(1-\beta))^{-ir}$. This proves the stable part. The unstable part is proved in a same way by considering the inverse iteration instead.
\end{proof}

\subsection{Regularity for H\"older continuous potential}

Suppose there are constants $K>0$ and $\alpha_0\in (0,1)$ such that our potential function $\varphi$ satisfies $|\varphi(x)-\varphi(y)|\leq Kd(x,y)^{\alpha_0}$ for all $x,y\in \mathbb{T}^2$. Our goal is to show that $\varphi$ has Bowen property at scale $100\epsilon$ on $\mathscr{G}(r)$ for any $0<r<1$.

\begin{lemma}
Given $(x,n)\in \mathscr{G}(r)$ and $y\in B_n(x,100\epsilon)$, we have
$d(G^k(x),G^k(y))\leq
100\gamma\epsilon((\lambda(1-\beta))^{-kr}+(\lambda(1-\beta))^{-(n-k-1)r}).$
\end{lemma}
\begin{proof}
As seen in the Lemma 3.7, by applying local product structure we are able to get $z\in \mathbb{T}^2$ such that $G^i(z)=W^s_{100\gamma\epsilon}(G^i(x))\cap W^u_{100\gamma\epsilon}(G^i(y))$ for $0\leq i \leq n-1$. By Lemma 6.1, we see immediately $d(G^k(x),G^k(z))\leq 100\gamma\epsilon(\lambda(1-\beta))^{-kr}$. To get the estimate for $d(G^k(y),G^k(z))$, we notice that for $\beta>0$ small enough, both $G^k(y)$ and $G^k(z)$ are in $B_{100\gamma\epsilon}(G^k(x))$. Because of the convexity of $B_{100\gamma\epsilon}(G^k(x))$, we can make the local unstable segment between $G^k(y)$ and $G^k(z)$ lie in $B_{100\gamma\epsilon}(G^k(x))$ for all $0\leq k \leq n-1$. A similar argument to the proof of Lemma 6.1 provides $d(G^k(x),G^k(z))\leq 100\gamma\epsilon(\lambda(1-\beta))^{-(n-1-k)r}$.
\end{proof}

With the help of Lemma 6.2, we are able to conclude the desired regularity condition for $\varphi$ (therefore for all H\"older continuous potential) over $\mathscr{G}(r)$, which is stated in the following proposition.
\begin{proposition}
$\varphi$ has Bowen property on $\mathscr{G}(r)$ at scale $100\epsilon$ for any $0<r<1$.
\end{proposition}

\begin{proof}
Given $(x,n)\in \mathscr{G}(r)$ and $y\in B_n(x,100\epsilon)$, from Lemma 6.2, H\"older continuity of $\varphi$ and $\lambda(1-\beta)>1$ we have
\begin{equation}
\begin{aligned}
|S_n\varphi(x)-S_n\varphi(y)| 
&\leq K\sum_{k=0}^{n-1}d(G^k(x),G^k(y))^{\alpha_0}\\ 
&\leq K(100\gamma\epsilon)^{\alpha_0}\sum_{k=0}^{n-1}((\lambda(1-\beta))^{-kr}+(\lambda(1-\beta))^{-(n-k-1)r})^{\alpha_0}.
\end{aligned}
\end{equation}

To estimate $\sum_{k=0}^{n-1}((\lambda(1-\beta))^{-kr}+(\lambda(1-\beta))^{-(n-k-1)r})^{\alpha_0}$, we have
\begin{equation}
\begin{aligned}
&\sum_{k=0}^{n-1}((\lambda(1-\beta))^{-kr}+(\lambda(1-\beta))^{-(n-k-1)r})^{\alpha_0}  \\
&\leq \sum_{k=0}^{n-1}(2(\max\{(\lambda(1-\beta))^{-kr},(\lambda(1-\beta))^{-(n-k-1)r}\}))^{\alpha_0} \\
&= 2^{\alpha_0}\sum_{k=0}^{n-1}(\max\{(\lambda(1-\beta))^{-kr},(\lambda(1-\beta))^{-(n-k-1)r}\})^{\alpha_0} \\
&\leq 2^{\alpha_0}\sum_{k=0}^{\infty}2(\lambda(1-\beta))^{\alpha_0}
= K_0 < \infty
\end{aligned}
\end{equation}

By (6.2) we have $|S_n\varphi(x)-S_n\varphi(y)|\leq KK_0(100\gamma\epsilon)^{\alpha_0} < \infty$.
\end{proof}

\subsection{Regularity for geometric t-potential}

In the uniformly hyperbolic case, the map $x\rightarrow E^{u}(x)$ is known to be H\"older continuous. Since $\log(x)$ function is Lipschitz continuous when $x$ is bounded away from $0$ and $\infty$, the geometric t-potential is automatically H\"older continuous.

Unfortunately, this argument does not extend to the non-uniformly hyperbolic Katok map. Though it is the limit of a sequence of Anosov diffeomorphisms, the respective H\"older exponent can be shown to blow up to $0$ by following a standard argument in \cite{Pes00}, Proposition 3.9. Therefore, the regularity for $\varphi_t(x)$ is not trivial.

Here we follow the spirit in the proof of regularity of geometric t-potential for Bonatti-Viana diffeomorphisms (see \cite{CFT17-1}). Compared to the dominated splittings, the additional technical difficulties are from the non-uniform expansion rate in $E^u$ over $E^s$.

The first few steps of the proof are similar to the Bonatti-Viana example. We will sketch these steps, explain on some technical details and underline the difference in the following steps for two proofs. 

\begin{proposition}
$\varphi^{geo}_G(x)$ satisfies Bowen property at scale $100\epsilon$ on $\mathscr{G}(r)$.
\end{proposition}

\begin{proof}
We first decompose $\varphi^{geo}_G(x):\mathbb{T}^2\rightarrow \mathbb{R}$ into $\psi' \circ E^u$. Here $E^u:x\rightarrow E^u(x)$ is a map from $\mathbb{T}^2$ to $G^1$, where $G^1$ is the one-dimensional Grassmannian bundle over $\mathbb{T}^2$ and $\psi'$ sends $E\in G^1$ to $-\log \vert DG(x)|_{E}\vert$. By identifying $G^1$ with $\mathbb{T}^2 \times \text{Gr}(1,\mathbb{R}^2)$ and writing out $\psi'$ as a composition of Lipschitz and smooth functions, it is proved in \cite{CFT17-1}, Lemma A.1 that given $G$ that is $C^{1+\alpha}$, the map $\psi'$ is H\"older continuous with exponent $\alpha$. 

We need to obtain a similar estimate for the distance in the tangent component $d_H(E^u(G^k(x)),E^u(G^k(y)))$ as in Lemma 6.2, where $d_H$ means the Hausdorff distance. This estimate, together with Lemma 6.2, gives us the Grassmannian bundle version of Lemma 6.2. By applying H\"older continuity of $\psi'$ and following the idea in Proposition 6.3, we are able to derive Bowen property for $\varphi^{geo}_G$.

For the remaining part of the proof we focus on proving the following

\begin{proposition}
For every $0<r<1$, There are $C\in \mathbb{R}$ and $\theta<1$ such that for every $(x,n)\in \mathscr{G}(r), y\in B_{100\epsilon}(x,n)$ and $0\leq k \leq n-1$, we have
$$d_{G_r}(E^u(G^k(x)),E^u(G^k(y))) \leq C(\theta^k+\theta^{n-1-k}).$$

Here, $d_{G_r}$ is the metric on $\text{Gr}(1,\mathbb{R}^2)$ defined as $d_{G_r}(E,E')=d_H(E\cap S^1,E'\cap S^1)$, where $d_H$ is the usual Hausdorff metric on compact subspace $S^1\subset \mathbb{R}^2$.

\end{proposition}

To prove this proposition, again by local product structure at scale $100\lambda(1+\beta)\gamma\epsilon$, we apply Lemma 3.7 to get $z\in \mathbb{T}^2$ such that $G^k(z)=W^s_{100\gamma\epsilon}(G^k(x))\cap W^u_{100\gamma\epsilon}(G^k(y))$ for $0\leq k \leq n-1$. We will estimate $d_{G_r}(E^u(G^k(x)),E^u(G^k(y)))$ in terms of $d_{G_r}(E^u(G^k(x)),E^u(G^k(z)))$ and $d_{G_r}(E^u(G^k(z),E^u(G^k(y)))$. Notice that $T_xW^u(x)=E^u(x)$ and $E^u$ is continuous, $W^u$ is $C^1$, so there is a constant $C$ such that $d_{G_r}(E^u(G^k(z),E^u(G^k(y))) \leq Cd(G^k(z),G^k(y)) \leq 100C\gamma\epsilon(\lambda(1-\beta))^{-(n-k-1)r}$. Therefore, to prove Proposition 6.5, it suffices to estimate the distance in $E^u$ along local stable leaves.

For $(x,n)\in \mathscr{G}(r)$ and $z\in W_{100\gamma\epsilon}^s(x)$, for any $0\leq k \leq n-1$ let $(e_{z,k}^i)_{i=1}^2$ be an orthonormal basis for $T_{G^k(z)}\mathbb{T}^2$ such that $E^s(G^k(z))=\text{span}(e_{z,k}^1)$. There is a way of choosing $(e_{z,k}^i)_{i=1}^2$ so that for every $k,i$, the map $z\rightarrow e_{z,k}^i$ is $K$-Lipschitz on $W_{100\gamma\epsilon}^s(x)$, where $K$ is independent of $x$,$n$,$i$ and $k$. This is because on small neighborhoods $U\subset \text{Gr}(1,\mathbb{R}^2)$, one can define a Lipschitz map $U\rightarrow \mathbb{R}\times \mathbb{R}$ that gives each element in $U$ an orthonormal basis. Since $\mathbb{T}^2$ is compact, we can choose this Lipschitz constant to be uniform in terms of $z$. On the other hand, since we are working on the local stable leaves and $(x,n)\in \mathscr{G}(r)$, from which we have an overall exponential contraction in $d_s$ under $G$, therefore have $K$ to be independent of $k$.

The fact that $z\rightarrow e_{z,k}^i$ is uniformly Lipschitz allows us to compute the term $d_{G_r}(E^u(G^k(x)),E^u(G^k(y)))$ using their coordinate representations in $e_{z,k}^i$. Let $\pi_{z,k}:T_{G^k(z)}\mathbb{T}^2\rightarrow \mathbb{R}^2$ be the coordinate representation in the basis of $e_{z,k}^i$. Let $A_k^z:\mathbb{R}^2 \rightarrow \mathbb{R}^2$ be the respective coordinate representation of $DG_{G^k(z)}$, i.e. $\pi_{z,k+1}\circ DG_{G^k(z)} = A_k^z\circ \pi_{z,k}$.

Now it suffices to show that $d_{G_r}(E_k^z,E_k^x)\leq C\theta^k$ where $E_k^x=\pi_{x,k}E^u(G^k(x))$. To show this, we need to study the dynamics of $A_k^z$ and $A_k^x$. Notice that by $E^s(G^k(z))=\text{span}(e_{z,k}^1)$, we have $A_k^z(Z)=Z$, where $Z=\mathbb{R}\times \{0\}\subset \mathbb{R}^2$. Let $\Omega$ be the set of subspaces $E\subset \mathbb{R}^2$ such that $Z\oplus E=\mathbb{R}^2$. Obviously $E_k^z\in \Omega$. To measure the $d_{G_r}(E_k^z,E_k^x)$, for $E\subset \Omega$, let $L_k^E:E_k^x\rightarrow Z$ be the linear map whose graph is $E$. From standard trigonometric computation we are able to get $\sin(d_{G_r}(E_k^x,E))\leq \Vert L_k^E \Vert$. If $\Vert L_k^{E_k^z} \Vert$ is decreasing exponentially fast in $k$, we know $\sin(d_{G_r}(E_k^x,E_k^z))$ will give approximately the value of $d_{G_r}(E_k^x,E_k^z)$, which is exactly what we want.

Now we want to estimate $\Vert L_k^{E_k^z} \Vert$ in terms of the dynamics of $A_k^z$ and $A_k^x$. Define $P:E_{k+1}^x\rightarrow A_k^zE_{k}^x$ to be the projection along $Z$, Lemma A.4 in \cite{CFT17-1} shows by another trigonometric argument that 
\begin{equation}
\begin{aligned}
L_{k+1}^{A_k^zE_k^z}+\text{Id} = (A_k^z|_Z \circ L_k^{E_k^z} \circ A_k^z|_{E_k^x}^{-1})\circ P.
\end{aligned}
\end{equation}
And in particular
\begin{equation}
\begin{aligned}
\Vert L_{k+1}^{A_k^zE_k^z} \Vert \leq \Vert A_k^z|_Z \Vert \cdot \Vert A_k^z|_{E_k^x}^{-1} \Vert \cdot \Vert P \Vert \cdot \Vert A_k^z|_{E_k^x}^{-1} \Vert + \Vert P-\text{Id} \Vert.
\end{aligned}
\end{equation}
By applying the H\"older continuity of $DG$, Lipschitz continuity of $e_{z,k}^i$ and $z\in W_{100\gamma\epsilon}^s(x)$ we get a constant $C$ independent of $x,z,n,i,k$ such that $\Vert A_k^z-A_k^x \Vert \leq C(100\gamma\epsilon)^{\alpha_0}(\lambda(1-\beta))^{-r\alpha_0}$. Therefore, we have 
\begin{equation}
\begin{aligned}
d_{G_r}(E_{k+1}^x,A_k^zE_{k}^x)=d_{G_r}(A_k^xE_{k}^x,A_k^zE_{k}^x) \leq C'(100\gamma\epsilon)^{\alpha_0}(\lambda(1-\beta))^{-r\alpha_0}
\end{aligned}
\end{equation}
for another constant $C'$ that is also independent of $x,z,n,i,k$. Take any $v\in E_{k+1}^x$ and look at the triangle formed by $v$, $Pv\in A_k^zE_k^x$ and $Pv-v=(P-\text{Id})v\in Z$. Then $\frac{\Vert Pv-v \Vert}{\Vert v \Vert}=\frac{\sin{\theta_1}}{\sin{\theta_2}}$, where $\theta_1$ is the angle between $v$ and $Pv$, $\theta_2$ is the angle between $Pv$ and $Pv-v$. We know $\theta_2$ is uniformly bounded away from $0$ and $\sin{\theta_1} \leq C''(100\gamma\epsilon)^{\alpha_0}(\lambda(1-\beta))^{-rk\alpha_0}$ for some constant $C''$ by (6.5). Therefore we have $\frac{\Vert Pv-v \Vert}{\Vert v \Vert} \leq C'''(100\gamma\epsilon)^{\alpha_0}(\lambda(1-\beta))^{-rk\alpha_0}$ for some constant $C'''$ independent of $x,z,n,i,k$. This gives the following:
\begin{equation}
\begin{aligned}
\Vert P-\text{Id} \Vert \leq C'''(100\gamma\epsilon)^{\alpha_0}(\lambda(1-\beta))^{-rk\alpha_0}
\end{aligned}
\end{equation}

Now we put (6.6) in (6.4) and get
\begin{equation}
\begin{aligned}
\Vert L_{k+1}^{A_k^zE_k^z} \Vert &\leq \Vert A_k^z|_Z \Vert \cdot \Vert A_k^z|_{E_k^x}^{-1} \Vert (1+C'''(100\gamma\epsilon)^{\alpha_0}(\lambda(1-\beta))^{-rk\alpha_0}) \Vert A_k^z|_{E_k^x}^{-1} \Vert \\
&+ C'''(100\gamma\epsilon)^{\alpha_0}(\lambda(1-\beta))^{-rk\alpha_0}.
\end{aligned}
\end{equation}
We write $\Vert A_k^z|_Z \Vert \cdot \Vert A_k^z|_{E_k^x}^{-1} \Vert$ as $P_k$.
There exists a constant $\lambda_0$ which satisfies the following properties:
\begin{enumerate}
    \item $\lambda_0\in (0,1)$.
    \item When $\chi(G^k(x))=1$, $P_k \leq \lambda_0 $.
\end{enumerate}
It is also easy to see that $P_i \leq 1$. Therefore, we have for any $(x,n)\in \mathscr{G}(r)$ and $z\in W_{100\gamma\epsilon}^s(x)$, $\prod_{i=0}^{j}P_i\leq {\lambda_0}^{(j+1)r}$ for $0\leq j \leq n-1$.

Write $\Vert L_{k+1}^{A_k^zE_k^z} \Vert$ as $D_k$, $C'''(100\gamma\epsilon)^{\alpha_0}$ as $Q$ and $(\lambda(1-\beta))^{-r\alpha_0}$ as $u$. We rewrite (6.7) into
\begin{equation}
\begin{aligned}
D_{k+1}\leq P_k(1+Qu^k)D_k+Qu^k.
\end{aligned}
\end{equation}

Up to this step there are no significant differences between the case of Bonatti-Viana diffeomorphisms and the Katok map. Nevertheless, for a dominated splitting example such as Bonatti-Viana diffeomorphisms, $P_k$ is strictly less than some constant $\lambda''<1$ for all $k$. Here for the Katok map, we don't have a uniform estimate on $P_k$. We will use $(x,n)\in \mathscr{G}(r)$ to help us get the desired exponential decay here.

Define $C_k:=\frac{D_k}{\nu^k}$, where $0<\nu<1$ is determined later and very close to $1$. Now (6.8) is turned into 

\begin{equation}
\begin{aligned}
C_{k+1}\leq \frac{P_k}{\nu}(1+Qu^k)C_k+Q\frac{u^k}{{\nu}^{k+1}}.
\end{aligned}
\end{equation}

We want to prove that $C_k$ is bounded for a suitable choice of $\nu$. We know $C_0=D_0\leq B$ for some $B>0$ by compactness of $\mathbb{T}^2$ and continuity of the unstable distribution.
Construct a sequence $\{F_k\}_{k\in \mathbb{N}\cup \{0\}}$ such that $F_0=B$ and
$$F_{k+1}=\begin{cases}
\frac{1}{\nu}(1+Qu^k)F_k+Q\frac{u^k}{{\nu}^{k+1}} \hfill \qquad \text{if } \chi(G^k(x))=0\\
\frac{\lambda_0}{\nu}(1+Qu^k)F_k+Q\frac{u^k}{{\nu}^{k+1}} \hfill \qquad \text{if } \chi(G^k(x))=1
\end{cases}$$

We notice that for different $(x,n)$ we will generate different sequence $\{F_k\}_{k\in \mathbb{N}\cup \{0\}}$. We want to show $F_k$ is uniformly bounded for all $(x,n)\in \mathscr{G}(r)$ with the fixed chosen $\nu$. This makes $C_k$ bounded by some number independent of $x,n,z,k$, as $C_k\leq F_k$ by the properties of $P_k$ and $\lambda_0$.

We first add some assumptions to $\nu$. We want $\frac{{u}^{\frac{r}{2}}}{\nu}<1$ and $\frac{{\lambda_0}^{\frac{r}{2}}}{\nu}<1$. Then choose two constants $\zeta>\frac{1}{\nu}$ and $\frac{\lambda_0}{\nu}<\eta<1$ such that $u<\nu\eta$ and ${\zeta}^{1-\frac{r}{2}}{\eta}^{\frac{r}{2}}<1$. We can choose such $\zeta$ and $\eta$ because $(\frac{u}{\nu})^{\frac{r}{2}}(\frac{1}{\nu})^{1-\frac{r}{2}}<1$ and $(\frac{\lambda_0}{\nu})^{\frac{r}{2}}(\frac{1}{\nu})^{1-\frac{r}{2}}<1$ by our assumption on $\nu$. Fix $\nu$ from now on.

There is an $N\in \mathbb{N}$ large enough such that when $k\geq N$, $\frac{1}{\nu}(1+Qu^k)<\zeta$ and $\frac{\lambda_0}{\nu}(1+Qu^k)<\eta$.

Now among all possible $(x,n)\in \mathscr{G}(r)$ with $n<N$, $F_k=F_k(x,n)$ is uniformly bounded by some $M>0$ for any $0\leq k \leq n$ due to the compactness of $\mathbb{T}^2$ and finiteness in the choice of $k,n$. We construct a new sequence$\{H_k\}_{k \geq N}$ such that $H_N=M$ and
$$H_{k+1}=\begin{cases}
\zeta H_k+\frac{Q}{\nu}(\frac{u}{\nu})^k \hfill \qquad \text{if } \chi(G^k(x))=0\\
\eta H_k+\frac{Q}{\nu}(\frac{u}{\nu})^k \hfill \qquad \text{if } \chi(G^k(x))=1
\end{cases}$$

Again it suffices to prove $H_k$ is uniformly bounded. We consider the large $k$ such that $k>\frac{2N}{r}$. By the choice of $k$ we have the following observation: $\sum_{i=N}^{k}\chi(F^i(x))>kr-N>\frac{rk}{2}$.

\begin{lemma}
For all $k>\frac{2N}{r}$, we have $H_k\leq M'$, where $M'$ is a constant independent of $x,n,z,k$.

\begin{proof}
Define $a_k=a_k(x,n):=\zeta(1-\chi(F^k(x)))+\eta \chi(F^k(x))$ for $k\geq N$. We have 
\begin{equation}
\begin{aligned}
H_{k+1}=a_kH_k+\frac{Q}{\nu}(\frac{u}{\nu})^k.
\end{aligned}
\end{equation}
By iterating (6.10) on $k$, we can write out $H_k$ explicitly for $k>N$ as follows
\begin{equation}
\begin{aligned}
H_k=(\prod_{i=N}^{k-1}{a_i})M+\frac{Q}{\nu}\sum_{j=N}^{k-1}((\frac{u}{\nu})^{j}\cdot \prod_{s=j+1}^{k-1}a_s). 
\end{aligned}
\end{equation}

Since $a_k\geq \eta$ and $u<\nu\eta$ by our assumption, we have

$$H_k\leq (\prod_{i=N}^{k-1}{a_i})M+\frac{Q}{\nu}(\frac{u}{\nu})^N\cdot (\prod_{s=N+1}^{k-1}a_s)\cdot \sum_{l=0}^{k-N-1}(\frac{u}{\nu} \cdot \frac{1}{\eta})^l \leq M+\frac{Q}{\nu}(\frac{u}{\nu})^N \cdot \frac{1}{\nu}\cdot S$$
where $S:=\sum_{l=0}^{\infty}(\frac{u}{\nu\eta})^l$. We can remove the term $\prod_{i=N+1}^{k-1}{a_i}$ as $\prod_{i=N+1}^{k-1}{a_i}\leq {\zeta}^{1-\frac{r}{2}}{\eta}^{\frac{r}{2}}<1$ for $\sum_{i=N}^{k}\chi(F^i(x))>\frac{rk}{2}$. By writing $M'=M+\frac{Q}{\nu}(\frac{u}{\nu})^N \cdot \frac{1}{\nu}\cdot S$, we get the result.
\end{proof}

\end{lemma}
As $H_k$ is uniformly bounded for $k>\frac{2N}{r}$, we have $H_k$ is uniformly bounded for all $k\geq N$. We know $F_k$ is bounded above by $H_k$ for $k\geq N$ and $M$ otherwise, hence uniformly bounded as well. Since we know $C_k$ is bounded above by $F_k$ from construction, Proposition 6.5 is finally proved, so is the Bowen property for $\varphi^{geo}_G$ on $\mathscr{G}(r)$ for all $0<r<1$.
\end{proof}

\section{Main Theorem}

\subsection{Verification of Theorem 1.1}

Now we have all the ingredients to prove Theorem 1.1. Before we state the proof, let us first briefly summarize on conditions of the parameters of the Katok map. We have $\beta=\frac{2\alpha}{2\alpha+1+\sqrt{4\alpha+1}}$ to be the slope of the invariant cone. To have enough expansion/contraction along the unstable/stable leaves, $\beta$, thus $\alpha$ needs to be sufficiently small. The perturbation also appears in a neighborhood of the origin with radius $r_0$ being small enough, as we require the local product structure at a scale greater than $500\lambda r_0$. In particular, these scales do not depend on the gap $P(\varphi)-\varphi(\underaccent{\bar}{0})$.

Now let us see how to apply Theorem 2.5 to deduce Theorem 1.1. We have the decomposition $(\mathscr{P}(r),\mathscr{G}(r),\mathscr{P}(r))$ for any $r>0$. To apply Theorem 2.5, we need to check all the conditions. Tail specification at scale $\epsilon$ is automatically satisfied for any $0<r<1$ by Proposition 3.6. Conditions for obstructions to expansivity is satisfied at scale $100\epsilon$ by Proposition 3.8. For potential function satisfying $\varphi(\underaccent{\bar}{0})<P(\varphi)$, which is definitely the case here, by Proposition 5.4 and argument in $\mathsection 5.3$, there is some $r'=r'(\varphi)>0$ such that $P(\mathscr{P}(r'),\varphi,\epsilon,100\epsilon)<P(\varphi)$. Finally, Proposition 6.3 gives us the Bowen property at scale $100\epsilon$ for H\"older continuous $\varphi$. Therefore, by taking $(\mathscr{P}(r'),\mathscr{G}(r'),\mathscr{P}(r'))$ to be the orbit decomposition, all the four conditions are verified and we conclude the proof.

\subsection{Verification of Theorem 1.2}

Now let's see how to deduce Theorem 1.2. In this case things are slightly different. Though the maps $G$ and $\widetilde{G}$ have the same dynamics, the geometric-t potentials are not the same function. Therefore, we are not able to fully copy the thermodynamic formalism of $G$ with $t\varphi^{geo}_G$ to derive the one for $\widetilde{G}$ with $t\varphi^{geo}$, where $\varphi^{geo}_G$ and $\varphi^{geo}$ are the geometric potentials associate to $G$ and $\widetilde{G}$.

Recall that $\widetilde{G}=\phi \circ G \circ \phi^{-1}$. Define $\mathscr{P}'(r)=\mathscr{S}'(r):=\{(x,n)\in \mathbb{T}^2 \times \mathbb{N}: (\phi^{-1}(x),n)\in \mathscr{P}(r)\}$ and $\mathscr{G}'(r):=\{(x,n)\in \mathbb{T}^2 \times \mathbb{N}: (\phi^{-1}(x),n)\in \mathscr{G}(r)\}$. By the fact that $\phi$ is identity outside $D_{r_1}$, it is not hard to see that
$$\mathscr{G'}(r)=\{(x,n):\frac{1}{i}S^{\widetilde{G}}_i\chi(x)\geq r \text{ and } \frac{1}{i}S^{\widetilde{G}}_i\chi(\widetilde{G}^{n-i}(x))\geq r \text{ for all }0\leq i \leq n \}$$
and
$$\mathscr{P'}(r)=\mathscr{S'}(r)=\{(x,n)\in \mathbb{T}^2 \times \mathbb{N}: \frac{1}{n}S^{\widetilde{G}}_n\chi(x)<r\}$$
where $S^{\widetilde{G}}_i\chi(x):=\sum_{j=0}^{i-1}\chi(\widetilde{G}^j(x))$.

Again, by repeating the discussion in $\mathsection 3$ and $\mathsection 4$, we know the orbit collections $(\mathscr{P}'(r),\mathscr{G}'(r),\mathscr{P}'(r))$ form an orbit decomposition for $\widetilde{G}$ with specification and expansivity. We also 
notice that the gap condition $t\varphi^{geo}(\underaccent{\bar}{0})<P(t\varphi^{geo};\widetilde{G})$ will provide us with the pressure gap with respect to $\widetilde{G}$. Therefore, to prove Theorem 1.2, we need to show $\varphi^{geo}$ has Bowen property over $\mathscr{r}'$ for any $0<r\leq 1$ and $t\varphi^{geo}(\underaccent{\bar}{0})<P(t\varphi^{geo};\widetilde{G})$ holds for all $t<1$.

We first deduce the regularity of $\varphi^{geo}$ from $\varphi^{geo}_{G}$. Since $\widetilde{G}=\phi \circ G \circ \phi^{-1}$ and $D\phi(E^u(x))=\widetilde{E}^u(\phi(x))$ where $\widetilde{E}^u(x)$ is the unstable distribution of $\widetilde{G}$ at $x$, for all $i\geq 0$ we have
\begin{equation}
    \begin{aligned}
    &\varphi^{geo}(\widetilde{G}^i(x))=-\log\vert D\widetilde{G}\vert_{\widetilde{E}^u(\widetilde{G}^i(x))} \vert =-\log\vert D(\phi \circ G \circ \phi^{-1})\vert_{\widetilde{E}^u(\widetilde{G}^i(x))} \vert \\
    &=-\log\vert D\phi\vert_{D(G\circ \phi^{-1})\widetilde{E}^u(\widetilde{G}^i(x))} \vert
    -\log \vert DG\vert_{D\phi^{-1}\widetilde{E}^u(\widetilde{G}^i(x))}\vert-\log \vert D\phi^{-1}\vert_{\widetilde{E}^u(\widetilde{G}^i(x))}\vert \\
    &=-\log\vert D\phi\vert_{D(G\circ \phi^{-1})\widetilde{E}^u(\widetilde{G}^i(x))} \vert
    -\varphi^{geo}_{G}(G^i(\phi^{-1}(x)))-\log \vert D\phi^{-1}\vert _{\widetilde{E}^u(\widetilde{G}^i(x))}\vert \\
    &=-\log\vert D\phi\vert_{DG(E^u(G^{i}(\phi^{-1}(x))))} \vert
    -\varphi^{geo}_{G}(G^i(\phi^{-1}(x)))-\log \vert D\phi^{-1}\vert _{\widetilde{E}^u(\widetilde{G}^i(x))}\vert \\
    &=-\log\vert D\phi\vert_{E^u(G^{i+1}(\phi^{-1}(x)))} \vert
    -\varphi^{geo}_{G}(G^i(\phi^{-1}(x)))-\log \vert D\phi^{-1}\vert _{\widetilde{E}^u(\widetilde{G}^i(x))}\vert .\\
    \end{aligned}
\end{equation}

We also have the following observation
\begin{equation}
    \begin{aligned}
    &0=-\log \vert D(\phi \circ \phi^{-1})\vert_{\widetilde{E}^u(\widetilde{G}^i(x))} \vert   \\
    &=-\log \vert D\phi \vert_{D\phi^{-1}\widetilde{E}^u(\widetilde{G}^i(x))}\vert 
    -\log \vert D\phi^{-1}\vert_{\widetilde{E}^u(\widetilde{G}^i(x))} \vert \\
    &=-\log \vert D\phi \vert_{E^u(G^i(\phi^{-1}(x)))}\vert 
    -\log \vert D\phi^{-1}\vert_{\widetilde{E}^u(\widetilde{G}^i(x))} \vert .\\
    \end{aligned}
\end{equation}

Therefore, by plugging (7.2) into (7.1), we have
\begin{equation}
    \begin{aligned}
    \varphi^{geo}(\widetilde{G}^i(x))=\log \vert D\phi^{-1}\vert _{\widetilde{E}^u(\widetilde{G}^{i+1}(x))}\vert
    -\varphi^{geo}_{G}(G^i(\phi^{-1}(x)))-\log \vert D\phi^{-1}\vert _{\widetilde{E}^u(\widetilde{G}^i(x))}\vert 
    \end{aligned}
\end{equation}
Now fix any $r\in (0,1]$. Given $(x,n)\in \mathscr{G}'(r)$ and $y$ such that $d(\widetilde{G}^i(x),\widetilde{G}^i(y))<\frac{100C\epsilon}{\kappa_0}$ for all $0\leq i\leq n-1$, where $\kappa_0$ is the normalizing constant in the definition of function $\phi$, $\kappa_0>1$ and $C=C(\alpha,r_0)$ is an expansion constant (see page 7). With the help of (7.3), we have
\begin{equation}
    \begin{aligned}
    &S_n^{\widetilde{G}}\varphi^{geo}(x)-S_n^{\widetilde{G}}\varphi^{geo}(y)
    =\sum_{i=0}^{n-1} (\varphi^{geo}(\widetilde{G}^i(x))- \varphi^{geo}(\widetilde{G}^i(y))) \\
    &=\sum_{i=0}^{n-1}(\log \vert D\phi^{-1}\vert _{\widetilde{E}^u(\widetilde{G}^{i+1}(x))}\vert
    -\varphi^{geo}_{G}(G^i(\phi^{-1}(x)))-\log \vert D\phi^{-1}\vert _{\widetilde{E}^u(\widetilde{G}^i(x))}\vert \\
    &-(\log \vert D\phi^{-1}\vert _{\widetilde{E}^u(\widetilde{G}^{i+1}(y))}\vert
    -\varphi^{geo}_{G}(G^i(\phi^{-1}(y)))-\log \vert D\phi^{-1}\vert _{\widetilde{E}^u(\widetilde{G}^i(y))}\vert )) \\
    &=\log \vert D\phi^{-1}\vert _{\widetilde{E}^u(\widetilde{G}^{n}(x))}\vert-\log \vert D\phi^{-1}\vert _{\widetilde{E}^u(\widetilde{G}^{n}(y))}\vert
    -\log \vert D\phi^{-1}\vert _{\widetilde{E}^u(x)}\vert
    +\log \vert D\phi^{-1}\vert _{\widetilde{E}^u(y)}\vert \\
    &+\sum_{i=0}^{n-1}(\varphi^{geo}_{G}(G^i(\phi^{-1}(y)))-\varphi^{geo}_{G}(G^i(\phi^{-1}(x)))).
    \end{aligned}
\end{equation}

Now we look at the last line of (7.4). Since we choose $(x,n)$ from $\mathscr{G}'(r)$, we know in particular that both $x$ and $\widetilde{G}^n(x)$ belong to $\mathbb{T}^2\setminus D_{100\gamma\epsilon+r_1}$. By definition of $y$, we know both $y$ and $\widetilde{G}^n(y)$ belong to $\mathbb{T}^2\setminus D_{r_1}$. Therefore, we know $\log \vert D\phi^{-1}\vert _{\widetilde{E}^u(\widetilde{G}^{n}(x))}\vert-\log \vert D\phi^{-1}\vert _{\widetilde{E}^u(\widetilde{G}^{n}(y))}\vert
-\log \vert D\phi^{-1}\vert _{\widetilde{E}^u(x)}\vert
+\log \vert D\phi^{-1}\vert _{\widetilde{E}^u(y)}\vert=0$ as $\phi^{-1}$ is identity in $\mathbb{T}^2\setminus D_{r_1}$. So to get the Bowen property of $\varphi^{geo}$, we only need to check if the remainder $\sum_{i=0}^{n-1}(\varphi^{geo}_{G}(G^i(\phi^{-1}(y)))-\varphi^{geo}_{G}(G^i(\phi^{-1}(x))))$ is bounded.

We know from definition that $(\phi^{-1}(x),n)\in \mathscr{G}(r)$. Therefore, to prove the result above, it suffices to show that $d(G^i(\phi^{-1}(x)),G^i(\phi^{-1}(y)))<100\epsilon$ because once this is proved, Proposition 6.4 will be immediately applicable. Notice that $G^i(\phi^{-1}(x))=\phi^{-1}\widetilde{G}^i(x)$, so $d(G^i(\phi^{-1}(x)),G^i(\phi^{-1}(y)))=d(\phi^{-1}\widetilde{G}^i(x),\phi^{-1}\widetilde{G}^i(y))$. By (3.1), we have $d(\phi^{-1}\widetilde{G}^i(x),\phi^{-1}\widetilde{G}^i(y))\leq \frac{\kappa_0}{C} d(\widetilde{G}^i(x),\widetilde{G}^i(y))<\frac{100C\kappa_0\epsilon}{C\kappa_0}=100\epsilon$. As a conclusion, we obtain the Bowen property of $\varphi^{geo}$ for $\widetilde{G}$ on $\mathscr{G}'(r)$ for any $0<r\leq 1$ at scale $\frac{100\epsilon}{\kappa_0}$ (the constant variation term in the Bowen property can differ in different $r$).

Now we verify that $t\varphi^{geo}(\underaccent{\bar}{0})<P(t\varphi^{geo};\widetilde{G})$ holds for all $t<1$. By Proposition 3.1(4), the lebesgue measure $m$ is preserved and ergodic under $\widetilde{G}$. Since the stable and unstable leaves are contained in cones with small angle, m has absolutely continuous conditional measure on unstable manifold. Since Lyapunov exponents of $m$ for $\widetilde{G}$ is nonzero by Proposition 3.1(3), $m$ is an SRB measure for $\widetilde{G}$. Therefore we have by \cite{LY85} 
$$h_m(\widetilde{G})=\lambda^+(m)=-\int{\varphi^{geo}dm},$$
where $\lambda^+$ refers to the positive Lyapunov exponent with respect to $m$. 

Since $-\int{\varphi^{geo}dm}>0$, we have
$$P(t\varphi^{geo};\widetilde{G})\geq P(t\varphi^{geo},m;\widetilde{G})=h_m(\widetilde{G})+t\int{\varphi^{geo}dm}=(1-t)\int{\varphi^{geo}dm}>0.$$

Therefore, if $t<1$, $P(t\varphi^{geo};\widetilde{G})>0=P(t\varphi^{geo},\delta_0)$. This conclude the proof of Theorem 1.2.

Further statistical properties of the Katok map are explored in \cite{PSZ16}, including exponential decay of correlations and central limit theorem for the unique equilibrium state. These are benefits brought by the inducing scheme technique applied there. Nevertheless, the uniqueness consequence on the equilibrium states for geometric t-potential are not as strong there. For a fixed $\widetilde{G}$, decreasing $t$ over a limit $t_0$ will destroy the positive recurrence of the normalized potential in the base. In this case, nothing can be said in terms of the uniqueness of equilibrium states. To fix this, the authors need to consistently narrow down the perturbed radius to make $t_0$ approach $-\infty$. See (P4) and Theorem 4.6 in \cite{PSZ17} for details.

Despite of the difference in conclusions, for geometric t-potentials, there are certain similarities regarding the spirit of two approaches. In \cite{PSZ16}, For $t_0<t<1$, there is an equilibrium state being unique among the measures lifted from the base of the inducing scheme and supported on the whole tower. In the case of the Katok map, the inducing time is simply the first recurrence to the base and the base is chosen to be an element in Markov partition induced by the original linear automorphism that is far away from perturbed region. By topological transitivity, the non-liftable measures have to distribute zero measures to each of these partition elements, which makes $\delta_0$ the only candidate. Pressure gap between $P(\varphi_t)$ and 0 will guarantee the equilibrium measure is chosen from the liftable measures, thus being unique. In our case, we prove the potential over orbit segments that spend enough time far away from the perturbed region are highly regular and strengthen the pressure gap result to all $t<1$ as our result is independent of the choice of Markov diagram.

\section{Global Weak Gibbs Property for equilibrium state} 

We exhibit a global weak Gibbs property for the unique equilibium state of potential functions in Theorem 1.1 and 1.2. For a continuous function $\varphi: X\rightarrow \mathbb{R}$, $\delta>0$ and $\mathscr{C}\subset X\times \mathbb{N}$, we say an invariant measure $\mu$ has Gibbs property at scale $\delta$ over $\mathscr{C}$ if there exists an $Q=Q(\delta,\mathscr{C})>1$ such that for every $(x,n)\in \mathscr{C}$, we have
$$Q^{-1}e^{-nP(\varphi)+S_n\varphi(x)}\leq \mu(B_n(x,\delta)) \leq Qe^{-nP(\varphi)+S_n\varphi(x)}.$$

If only the left (right) inequality holds, we say $\mu$ has lower (upper) Gibbs property at scale $\delta$ over $\mathscr{C}$.

For the unique equilibrium state of orbit decomposition satisfying all assumptions in Theorem 2.5, in \cite{CT16}, the authors deduce a version of upper Gibbs property in terms of two-scale estimate over $X\times \mathbb{N}$ and lower Gibbs property over $\mathscr{G}^M$. In Katok map, since all orbit segments have specification at any scales, it is possible to prove a weak lower Gibbs property on $X\times \mathbb{N}$. 

We fix the potential function $\varphi$ to be any potential satisfying the condition of Theorem 1.1 or 1.2 (geometric-t potential with $t<1$ or H\"older continuous potential with $P(\varphi)-\varphi(\underaccent{\bar}{0})$) and $\mu$ to be the respective unique equilibrium state. We just discuss on $G$ as all the properties can be directly referenced from earlier results in the paper and $\widetilde{G}$ share all those properties according to $\mathsection 7$. We also fix an appropriate $r>0$ such that $(\mathscr{P}(r), \mathscr{G}(r), \mathscr{P}(r))$ is the desired orbit decomposition for $\varphi$. Recall the process of constructing the equilibrium measure $\mu$ is as follows. For each $n\in \mathbb{N}$, let $E_n\subset X$ be a maximizing $(n,5\epsilon)$-separated set for $\Lambda(X,n,5\epsilon)$, where $\epsilon$ is the same as before. Consider the measures
$$\nu_n:=\frac{\sum_{x\in E_n}e^{S_n\varphi(x)}\delta_x}{\sum_{x\in E_n}e^{S_n\varphi(x)}},$$
$$\mu_n:=\frac{1}{n}\sum_{i=0}^{n-1}(G^i)_{*}\nu_n.$$

By the second part of the proof of variational principle in \cite{Wal82} and the fact that $\epsilon$ is much smaller than the expansive constant for $G$, we have any weak* limit of $\{\mu_n\}$ to be an equilibrium state. By uniqueness of the equilibrium state, we know $\mu_n$ converges in weak* topology. See Lemma 4.14 and 6.12 in \cite{CT16}.
\subsection{Global Weak Lower Gibbs Property}

We have the following weak version of lower Gibbs property for $\mu$ that applies to all orbits with the Gibbs constant decaying subexponentially. 

\begin{proposition}
There is $Q=Q(\epsilon)>0$ such that for every $(x,n)\in X\times \mathbb{N}$, we have
$$\mu(B_n(x,6\epsilon))\geq Qe^{-\zeta(n)}e^{-nP(\varphi)+S_n\varphi(x)}$$
\end{proposition}

\begin{proof}
For any $(x,n)\in X\times \mathbb{N}$, we estimate $\mu(B_n(x,6\epsilon))$ using $\nu_s(G^{-k}(B_n(x,6\epsilon)))$ with $s\gg n$, $k\gg n$ and $s-k \gg n$. The technique is similar to the one in Lemma 4.16 \cite{CT16}, and here we do estimates over all orbit segments in the homeomorphism case using global specification. By Proposition 4.10 in \cite{CT16}, there is $T,L>0$ such that 
\begin{equation}
\Lambda(\mathscr{G},12\epsilon,m)>e^{-L}e^{mP(\varphi)}.    \end{equation}
for all $m\geq T$. Then for every $m\geq T$ we can find an $(m,12\epsilon)$-separated set $E_m'\subset \mathscr{G}_m$ such that
\begin{equation}
\sum_{x\in E_m'}e^{S_n\varphi(x)}\geq e^{-L}e^{mP(\varphi)}
\end{equation}

To estimate $\nu_s(G^{-k}(B_n(x,6\epsilon)))$, we use the specification of $G$ at scale $\epsilon$. Suppose the transition time $\tau=\tau(\epsilon)$. We fix $s$ and $k$. Without loss of generality we assume $k \gg T+\tau$ and $s-k-n \gg T+\tau$. We construct a map $\pi: E_{k-\tau}'\times E_{s-k-n-\tau}'\rightarrow E_s$ as follows. 

For $u=(u_1,u_2)\in E_{k-\tau}'\times E_{s-k-n-\tau}'$, by specification at scale $\epsilon$, there is a $y=y(u)$ such that $y\in B_{k-\tau}(u_1,\epsilon)$, $G^{k}(y)\in B_{n}(x,\epsilon)$ and $G^{k+n+\tau}(y)\in B_{s-k-n-\tau}(u_2,\epsilon)$. By definition of $E_s$, we can define $\pi(u)\in E_s$ such that $d_s(\pi(u),y(u))<5\epsilon$. Since $E'_{k-\tau}$ and $E'_{s-k-n-\tau}$ are $(k-\tau,12\epsilon)$-separated and $(s-k-n-\tau,12\epsilon)$-separated respectively, if $u'\neq u''$ for some $u'=(u_1',u_2')$, $u''=(u_1'',u_2'')$ that both belong to $E_{k-\tau}'\times E_{s-k-n-\tau}'$, we have $d_s(\pi(u'),\pi(u''))>12\epsilon-2(5\epsilon+\epsilon)=0$. Therefore $\pi$ is injective and by definition we have $\pi(u)\in G^{-k}(B_n(x,6\epsilon))$. By applying Bowen property for $G$ with $\varphi$ over $\mathscr{G}$ at scale $100\epsilon$, we have
\begin{equation}
    \Phi_0(\pi(u),s)-\Phi_0(u_1,k-\tau)-\Phi_0(x,n)-\Phi_0(u_2,s-k-n-\tau)\geq -4{\tau}\vert \varphi \vert-2K-\zeta(n).
\end{equation}
where $\Phi_0(x,n):=S_n\varphi(x)$, $\vert \varphi \vert:=\sup\{\vert \varphi(x) \vert:x\in \mathbb{T}^2\}$, $K$ is the constant in Bowen property and $\zeta$ is the variation term as in Definition 5.5. 

We estimate $\nu_s(G^{-k}(B_n(x,6\epsilon)))$ from below. By Lemma 4.11 in \cite{CT16}, since $\varphi$ has Bowen property over $\mathscr{G}(r)$ at scale $100\epsilon$ and $P(\mathscr{P}(r),\varphi,\epsilon,100\epsilon)<P(\varphi)$, there is a constant $C>0$ independent of $s$ such that $\sum_{z\in E_s}e^{\Phi_0(z,s)}\leq Ce^{sP(\varphi)}$. We have
$$
\begin{aligned}
&\nu_s(G^{-k}(B_n(x,6\epsilon))) 
\geq C^{-1}e^{-sP(\varphi)}\sum_{u\in E_{k-\tau}'\times E_{s-k-n-\tau}'}e^{\Phi_0(\pi(u),s)}
\\
&\geq C^{-1}e^{-sP(\varphi)}e^{-\zeta(n)-4{\tau}\vert \varphi \vert-2K}(\sum_{u_1\in E_{k-\tau}'}e^{\Phi_0({u_1,k-\tau})})(\sum_{u_2\in E_{s-k-n-\tau}'}e^{\Phi_0({u_2,s-k-n-\tau})})
e^{\Phi_0(x,n)}\\
&\geq C^{-1}e^{-sP(\varphi)}e^{-\zeta(n)-4{\tau}\vert \varphi \vert-2K}(e^{-L}e^{(k-\tau)P(\varphi)})(e^{-L}e^{(s-k-n-\tau)P(\varphi)})e^{\Phi_0(x,n)}\\
&=(C^{-1}e^{-2K}e^{-2L}e^{-4{\tau}\vert \varphi \vert}e^{-2\tau P(\varphi)})(e^{-\zeta(n)}e^{-nP(\varphi)+\Phi_0(x,n)})\\
&=C_1e^{-\zeta(n)}e^{-nP(\varphi)+\Phi_0(x,n)}
\end{aligned}
$$
The first inequality follows from the fact that the map $\pi$ is injective as well as $\sum_{z\in E_s}e^{\Phi_0(z,s)}\leq Ce^{sP(\varphi)}$. The second inequality follows from (8.3). The third inequality follows from (8.2). In the last equality the constant $C_1$ is just an rewriting of $C^{-1}e^{-2K}e^{-2L}e^{-4{\tau}\vert \varphi \vert}e^{-2\tau P(\varphi)}$ and we can see $C_1$ is only dependent on $\epsilon$, in particular, independent of $s$ or $k$. Therefore, by summing over $k$, we have
$$\mu_s(G^{-k}(B_n(x,6\epsilon)))=\frac{1}{s}\sum_{i=0}^{s-1}((G^i)_{*}\nu_{s})(B_n(x,6\epsilon)) \geq C_1e^{-\zeta(n)}e^{-nP(\varphi)+\Phi_0(x,n)}.$$
which leads to the statement of the proposition thus completes the proof.
\end{proof}

We observe that the $\epsilon>0$ used throughout the paper could be made arbitrarily small and Proposition 8.1 holds at all scales with different $Q$. Together with the fact that $\lim_{n\rightarrow \infty}{\frac{\zeta(n)}{n}}=0$, we have
\begin{equation}
\lim_{\epsilon \rightarrow 0}\liminf_{n \rightarrow \infty}\inf_{x\in \mathbb{T}^2}(\frac{1}{n}\log(\mu(B_n(x,\epsilon)))+\int{(P(\varphi)-\varphi})d{\delta_{x,n}})\geq 0.
\end{equation}
where $\delta_{x,n}=\frac{1}{n}\sum_{i=0}^{n-1}\delta_{G^i(x)}$.

Since $\varphi$ is continuous, (8.4) gives the definition of $P(\varphi)-\varphi$ being an lower-energy function for $\mu$ in \cite{PS05}, Definition 3.2. The existence of lower-energy function for $\mu$ is crucial in deriving the lower large deviation principle for $\mu$. We will give detailed definitions and explanations in $\mathsection 8.4$. 

\subsection{Upper Gibbs Property}

Proposition 4.21 in \cite{CT16} estimates the upper Gibbs property for $\mu$ over $X \times \mathbb{N}$ in terms of $\Phi_{6\epsilon}$ (see the definition at the end of $\mathsection 2.1$). It says there exists $Q'=Q'(\epsilon)$ such that for every $(x,n)\in X\times \mathbb{N}$, we have 
\begin{equation}
\mu(B_n(x,6\epsilon))\leq Q'e^{-nP(\varphi)+\Phi_{6\epsilon}(x,n)}.    
\end{equation}

By definition of $\zeta(n)$, we have $\Phi_{6\epsilon}(x,n)\leq \Phi_{0}(x,n)+\zeta(n)$, thus from (8.5) we have
\begin{equation}
\mu(B_n(x,6\epsilon)) \leq Q'e^{-nP(\varphi)+\Phi_{0}(x,n)+\zeta(n)}.  
\end{equation}

Similar to (8.4), from (8.6) and the fact that $\lim_{n\rightarrow \infty}\frac{\zeta(n)}{n}=0$ we have
\begin{equation}
\lim_{\epsilon \rightarrow 0}\limsup_{n \rightarrow \infty}\sup_{x\in \mathbb{T}^2}(\frac{1}{n}\log(\mu(B_n(x,\epsilon))+\int{(P(\varphi)-\varphi})d{\delta_{x,n}})\leq 0.
\end{equation}

The inequality (8.7) and continuity of $\varphi$ shows $P(\varphi)-\varphi$ is an upper-energy function for $\mu$ according to Definition 3.4 in \cite{PS05}. Similar to the case of lower-energy function, it plays an essential role in deriving the upper large deviation principle for $\mu$. We will clarify all the details in $\mathsection 8.4$ as well.

\subsection{Entropy Density}

We say $(X,f)$ has the property of entropy density (of ergodic measures) if for any invariant measure $\mu$ and any $\eta>0$, there is an ergodic measure $\nu$ such that $D(\mu,\nu)<\eta$ and $\vert h_{\mu}(f)-h_{\nu}(f) \vert <\eta$, where $D$ is a metric over the space of measures on $X$ compatible with the weak* topology. 

In our case, there are several approaches to give rise to the entropy density property of $(\mathbb{T}^2,\widetilde{G})$. First from Proposition 3.1(1) we know $\widetilde{G}$ is homeomorphically conjugate to $f_A$, which is a transitive Anosov diffeomorphism. From classic results we know that $f_A$ has the entropy density property, which immediately implies the desired result on $\widetilde{G}$.

Here we point out many examples of non-uniformly hyperbolic diffeomorphisms are not conjugate to transitive Anosov systems. Therefore, we will also sketch a proof that uses Gorodetski and Pesin's results from \cite{GP17}, which relies on the properties of hyperbolic periodic orbits and is potentially more applicable to other non-uniformly hyperbolic settings. 

In \cite{GP17}, the authors define hyperbolic periodic points $p,q\in X$ to be homoclinically related if the stable manifold of the orbit of $p$ intersects transversely with the unstable manifold of the orbit of $q$ and vice versa. Denote by $\mathscr{H}(p)$ the closure of the set of all hyperbolic periodic points homoclinically related to $p$ and $s(p)$ the topological dimension of the stable manifold of $p$. Two assumptions concerning $\mathscr{H}(p)$ are added to $(X,f)$, if for any hyperbolic periodic point $p$:

(H1) For any hyperbolic periodic point $q\in \mathscr{H}(p)$ with $s(q)=s(p)$, q and p are homoclinically related.

(H2) $\mathscr{H}(p)$ is isolated. This means that there is an open neighborhood $U(\mathscr{H}(p))$ of $\mathscr{H}(p)$ such that $\mathscr{H}(p)=\bigcap_{n\in \mathbb{Z}}{f^n(U(\mathscr{H}(p)))}$. 

\noindent Then the authors conclude that $\mathscr{M}_p^e$ is entropy dense in $\mathscr{M}_p$, where $\mathscr{M}_p$ is the set of all invariant hyperbolic measures supported on $\mathscr{H}(p)$ for which the number of negative Lyapunov exponents at almost every point is exactly $s(p)$ and $\mathscr{M}_p^e\subset \mathscr{M}_p$ is the set of ergodic ones.

In the Katok map, all the periodic points not equal to origin are hyperbolic. Moreover, by Lemma 3.3, every pair of hyperbolic periodic points $(p,q)$ are homoclinically related with $s(p)=s(q)=1$. By Proposition 3.1(1), hyperbolic periodic points are dense for $f_A$ in $\mathbb{T}^2$, thus dense for $\widetilde{G}$ in $\mathbb{T}^2$. Therefore $\mathscr{H}(p)=\mathbb{T}^2$ for all hyperbolic periodic $p$ and thus both (H1) and (H2) hold in the case of the Katok map. Therefore, by applying the above result, the set of all hyperbolic ergodic measures is entropy dense in the set of all hyperbolic invariant measures. 

To prove the entropy density for $\widetilde{G}$, by Proposition 3.1(3), it suffices to prove a linear combination of $\delta_0$ and any invariant hyperbolic measure can be approximated in distance and entropy by ergodic ones. Choose any invariant hyperbolic measure $\nu$ and $0<a<1$ and consider $\nu_a:=a\delta_0+(1-a)\nu$. 

According to the differential system that generates $G$, we have the following observation:
$$\frac{d(s_1s_2)}{dt}=-s_1s_2\psi(s_1^2+s_2^2)\log\lambda+s_2s_1\psi(s_1^2+s_2^2)\log\lambda=0.$$
That is to say, when the orbit stay in the single local chart, $s_1s_2$ is a constant. Moreover, $s_1(t)$ is non-decreasing in $D_{r_1}$ and strictly increasing except for $W^s_{r_1}(0)$, with $s_2((t)$ being non-increasing in $D_{r_1}$ and strictly decreasing except for $W^u_{r_1}(0)$. Given any $x\in D_{r_1}$ with eigen-coordinate being $x_1,x_2$ and $x_1\neq 0$ (otherwise $x$ will converge to origin), we know locally the orbit of $x$ will be on $s_1s_2=\rho$. By evaluating on $\frac{ds_1}{dt}=s_1\psi(s_1^2+s_2^2)\log\lambda \geq s_1\psi(2s_1s_2)\log\lambda=s_1\psi(2\rho)\log\lambda$, we get an upper bound $T(\rho)\approx \frac{r_1^2}{\rho {\lambda}^{\psi(2\rho)}}$ for the time that the orbit of $x$ will spend in $D_{r_1}$ before moving out. Since $G=f_A$ outside, after $\{G^n(x)\}$ reaches $\mathbb{T}^2 \setminus D_{r_1}$, it will stay there for at least $\frac{1}{r_1\lambda}$ times. After that, the orbit will wrap around the torus and $s_1s_2$ will change.

For any $n\in \mathbb{N}$ and $\delta>0$, by choosing $s_1s_2$ small, we are able to find an orbit segment with length $n$ that stays close to the origin within distance $\delta$. By applying specification at a fixed scale $\delta' \ll r_0$ and make $\delta \ll \delta'$ and $n\rightarrow \infty$, we are able to get a sequence of periodic point $\{p_n\}_{n\geq 1}$ such that for each $n$, $p_n$ spends more than $\frac{n^2-1}{n^2}q_n$ times in a neighborhood $U_n$ of origin whose diameter is less than $\frac{\delta'}{n}$, where $q_n$ is the period of $p_n$. The same result applies to $\widetilde{G}$. 

From definition of $p_n$ we have $D(\delta_0,\delta_{p_n})\rightarrow 0$ when $n\rightarrow \infty$, where $\delta_{p_n}$ is the periodic measure supported on $p_n$. By shrinking $\delta'$ if necessary, without loss of generality we assume $D(\delta_0,\delta_{p_n})<\frac{1}{n}$. Consider $\mu_{n,a}:=a\delta_{p_n}+(1-a)\nu$. Recall $\nu$ is a fixed invariant hyperbolic measure. For each $n>0$, $\mu_{n,a}$ is invariant and has zero measure at origin, thus hyperbolic. Therefore, there is an ergodic $\nu_{n,a}$ such that $D(\nu_{n,a},\mu_{n,a})<\frac{1}{n}$ and $\vert h_{\mu_{n,a}}(\widetilde{G})-h_{\nu_{n,a}}(\widetilde{G}) \vert<\frac{1}{n}$. Since $D(\nu_{a},\mu_{n,a}) < \frac{a}{n}$ and $h_{\nu_{a}}(\widetilde{G})=h_{(1-a)\nu}(\widetilde{G})=h_{\mu_{n,a}}(\widetilde{G})$, we have
$D(\nu_a,\nu_{n,a})<\frac{2}{n}$ and $\vert h_{\nu_{a}}(\widetilde{G})-h_{\nu_{n,a}}(\widetilde{G}) \vert<\frac{1}{n}$, which gives us the desired results for entropy density.

\subsection{Large Deviations Principle}

In this section we combine the results in $\mathsection 8$ to deduce the large deviations principle of $\mu$. Large deviations principle describes the exponential decay of the measure of points whose space average differs from the time average by a certain distance. In terms of estimating from below or above, we have the definition for upper and lower large deviations principle.

\begin{definition}
Let $\mu$ be the equilibrium state for potential $\varphi$. We say that $\mu$ has upper large deviations principle if for any continuous $\widetilde{f}:\mathbb{T}^2\rightarrow \mathbb{R}$ and any $\delta>0$, we have
$$\limsup_{n\rightarrow \infty}\frac{1}{n}\log\mu\left\{x:\vert \frac{S_n\widetilde{f}(x)}{n}-\int \widetilde{f} d\mu \vert\geq \delta \right\}\leq -q(\delta),$$
where $q(\delta)$ is the rate function given by
$$q(\delta):=P(\varphi)-\sup\left\{h_{\nu}(\widetilde{G})+\int \varphi d\nu: {\nu\in \mathscr{M}_{\widetilde{G}}(\mathbb{T}^2)}, \quad \vert \int \widetilde{f} d\mu-\int \widetilde{f} d\nu \vert \geq \delta\right\}.$$
or $q(\delta)=\infty$ when there is no such measure $\nu$. 

\end{definition}

Similarly, the lower large deviations principle holds when we have a liminf in place of limsup, $> \delta$ in place of $\geq \delta$ and $\geq$ in place of $\leq$ for the whole inequality. If both lower and upper large deviations hold for a fixed $\widetilde{f}$, the statement above is known as level-1 large deviations priciple. If they hold for all $\widetilde{f}$, the statement above is equivalent to level-2 large deviations principle.

The traditional definition (see for example Definition 5 in \cite{You90}) of large deviations principle requires that $q(\delta)$ should be lower semi-continuous. Here it is true since the entropy map is upper semi-continuous for expansive map and $\widetilde{G}$ is continuous.

Now let us prove Theorem 1.3. We continue to use the same notation of $\varphi$ and $\mu$ as in the beginning of this section. It suffices to prove the lower and upper large deviations principle for $\mu$. In $\mathsection 8.1$ we obtain a weak version of lower Gibbs property for $\mu$. In particular, in the end we show that $P(\varphi)-\varphi$ is a lower-energy function for $\mu$. In $\mathsection 8.3$ we prove the entropy density of ergodic measures. By applying Theorem 3.1 in \cite{PS05}, we get the lower large deviations principle for $\mu$.

In $\mathsection 8.2$ we have the respective weak version of upper Gibbs property for $\mu$, which leads to $P(\varphi)-\varphi$ being an upper-energy function for $\mu$. As entropy map is upper semi-continuous, by Theorem 3.2 in \cite{PS05}, we have the upper large deviations principle for $\mu$.

\section{Multifractal spectra}

We now carry out multifractal analysis on the Katok map $\widetilde{G}$ with potential function $\varphi_t:=t\varphi^{geo}$ by studying the level sets of Lyapunov exponents. Multifractal analysis measures the size of the set with the same given local asymptotic quantity associated to the dynamical system. In dynamics, the size usually refers to the Hausdorff dimension entropy or pressure. In our case, we take the local asymptotic quantities to be Birkhoff average and estimate the level set in terms of entropy and dimension. The goal of this section is to prove Theorem 1.4.

\subsection{General background and outline of the proof}

We begin with a few definitions (see also \cite{Cli10}). The non-negative (forward) Lyapunov exponents at all Lyapunov regular points $x$ is the (forward) Birkhoff average of $-\varphi^{geo}$:

$$\chi^{+}(x)=\lim_{n\rightarrow \infty}\log\frac{\Vert D\widetilde{G}^n|_{E^u(x)} \Vert}{n}=\lim_{n\rightarrow \infty}\frac{S_n(-\varphi^{geo})(x)}{n}.$$

Similarly we can define $\chi^{-}(x)$ by simply making $n$ in the definition of $\chi^{+}(x)$ go to $-\infty$. In the two dimensional case, if $\chi^{+}(x)=\chi^{-}(x)$, we say that the point $x$ is Lyapunov regular. To study the level sets of non-negative Lyapunov exponents we give the following natural definition:

$$L(\beta):=\Big\{ x\in \mathbb{T}^2: x \text{ is Lyapunov regular and }\chi^{+}(x)=\beta \Big\}.$$

We also define $\mathscr{P}(t):=P(\varphi_t)$ and $\mathscr{E}(\alpha):=\inf_{t\in \mathbb{R}}(\mathscr{P}(t)-t\alpha)$, which is the Legendre transform of $\mathscr{P}$. We know $\mathscr{P}$ is convex, so it has left and right derivative $D^{-}\mathscr{P}(t)$, $D^{+}\mathscr{P}(t)$ at each $t\in \mathbb{R}$. 

We concentrate on the Hausdorff dimension and topological entropy of $L(-\alpha)$, denoted by $d_H(-\alpha)$ and $h(-\alpha)$. Note that $-\alpha$ is the value of the non-negative Lyapunov exponent, thus all the $\alpha$ that appear in this section are non-positive (Also notice that this is not the same $\alpha$ as in $\mathsection 3-\mathsection 7$). We also have different use of $h$ in this section. For $E\subset \mathbb{T}^2$, $h(E)$ means its topological entropy in the sense of Bowen\cite{Bow73}. For $\mu\in \mathcal{M}(\widetilde{G})$, $h(\mu)$ represents its measure-theoretic entropy. For $\beta\in \mathbb{R}$, $h(\beta)$ is defined as above.

Now we sketch the proof of Theorem 1.4. We first notice that the fact of Lebesgue measure $m$ being an SRB measure brings us the phase transition at $t=1$ for $\widetilde{G}$ with $\varphi_t$, which says there's a gap between $\alpha_2:=D^{-}\mathscr{P}(1)$ and $D^{+}\mathscr{P}(1)$, which is simply $0$. Define $\alpha_1:=\lim_{t\rightarrow -\infty}D^{+}\mathscr{P}(t)$. For the entropy spectrum, by applying the uniqueness result of the equilibrium state in  Theorem 1.2, we have a complete picture for $\alpha\in (\alpha_1,\alpha_2)$ from Theorem 3.1.1 in \cite{Cli10}, which says that $h(-\alpha)=\mathscr{E}(\alpha)$ for all such $\alpha$ and $L(-\alpha)=\emptyset$ for $\alpha<\alpha_1$ or $\alpha>0$. For $\alpha\in [\alpha_2,0)$, we need to show that $L(-\alpha)$ is non-empty. This will enable us to apply Theorem 3.5 in \cite{TV03} and get $h(-\alpha)=\mathscr{E}(\alpha)$ for these $\alpha$. In conclusion, we know $h(-\alpha)=\mathscr{E}(\alpha)$ for all $\alpha\in (\alpha_1,0)$.

To prove $L(-\alpha)$ is non-empty for $\alpha\in [\alpha_2,0)$, we follow the construction in \cite{BG14}.  The idea is to construct a sequence of invariant subsets with well-known dimension estimates that asymptotically consumes all the pressure. Essentially, we construct a nested sequence of basic sets $\{\widetilde{\Lambda}_i\}_{i\in \mathbb{N}}$ such that $\widetilde{\Lambda}_i\subset \widetilde{\Lambda}_{i+1}$ and $\mathscr{P}_{\widetilde{\Lambda}_i}(t)\nearrow \mathscr{P}(t)$, where $\mathscr{P}_{\widetilde{\Lambda}_i}(t)$ means the topological pressure of $\varphi_t$ over $\widetilde{\Lambda}_i$ and by basic sets, we refer to locally maximal compact transitive $\widetilde{G}$-invariant hyperbolic sets. Then due to the thermodynamic formalism of basic sets and smoothness of the pressure function $\mathscr{P}_{\widetilde{\Lambda}_i}(t)$, for each $\alpha\in [\alpha_2,0)$, there is some $N_{\alpha}\in \mathbb{N}$ such that $L(-\alpha)\cap \widetilde{\Lambda}_n\neq \emptyset$ for all $n\geq N_{\alpha}$, which provides us the desired result.

To estimate $d_H(-\alpha)$ for $\alpha \in (\alpha_1,0)$, we also rely on the construction of $\{\widetilde{\Lambda}_i\}_{i\in \mathbb{N}}$ above. We estimate $d_H(-\alpha)$ from below in terms of the Hausdorff dimension of $\lim_{n\rightarrow \infty}L(-\alpha)\cap \widetilde{\Lambda}_n$. Define $\mathscr{E}_{\widetilde{\Lambda}_i}(\alpha):=\inf_{t\in \mathbb{R}}(\mathscr{P}_{\widetilde{\Lambda}_i}(t)-t\alpha)$ as the Legendre transform of $\mathscr{P}_{\widetilde{\Lambda}_i}(t)$. The Hausdorff dimension of the Lyapunov spectrum for the basic set is well-known: It should be $\dim_H(\widetilde{\Lambda}_i\cap L(-\alpha))=\frac{2\mathscr{E}_{\widetilde{\Lambda}_i}(\alpha)}{-\alpha}$. Moreover, if $\lim_{i\rightarrow \infty}\mathscr{E}_{\widetilde{\Lambda}_i}(\alpha)=\mathscr{E}(\alpha)$ for all $\alpha \in (\alpha_1,0)$, we immediately get $\frac{2\mathscr{E}(\alpha)}{-\alpha}$ to be the desired lower bound of $d_H(-\alpha)$. However, in the homeomorphism case, this statement concerning the Lyapunov spectrum of basic set does not seem to exist in any of the references available. To avoid ambiguity, we apply the Main Theorem from \cite{You82}, which says $\dim_H(\widetilde{\Lambda}_i\cap L(-\alpha)) \geq \frac{2\mathscr{E}_{\widetilde{\Lambda}_i}(\alpha)}{-\alpha}$. Therefore, we still get $\frac{2\mathscr{E}(\alpha)}{-\alpha}$ to be the lower bound of $d_H(-\alpha)$, which concludes Theorem 1.4.  

\subsection{Entropy spectrum for $\alpha \in (\alpha_1, \alpha_2)$}
We first show that $h(-\alpha)=\mathscr{E}(\alpha)$ for $\alpha\in (\alpha_1, \alpha_2)$. Recall that $\alpha_1=\lim_{t\rightarrow -\infty}D^{+}\mathscr{P}(t)$, $\alpha_2=D^{-1}\mathscr{P}(1)$ and $\mathscr{E}(\alpha)=\inf_{t\in \mathbb{R}}(\mathscr{P}(t)-t\alpha)$. To obtain this result, we need the uniqueness of equilibrium state for $\phi_t$ derived in Theorem 1.2 and apply Theorem 3.1.1 in \cite{Cli10}.

\begin{proposition}
\begin{enumerate}
    \item $\mathscr{P}$ is the Legendre transform of the Birkhoff spectrum. In other words, 
$$\mathscr{P}(t)=\sup_{\alpha\in \mathbb{R}}(h(-\alpha)+t\alpha).$$

    \item $L(-\alpha)=\emptyset$ for every $\alpha<\alpha_{1}$ and every $\alpha>0$.
    
    \item $h(-\alpha)$ has domain $(\alpha_1, \alpha_2)$ and is the Legendre transform of $\mathscr{P}$, which means:
    $$h(-\alpha)=\inf_{t\in \mathbb{R}}(\mathscr{P}(t)-t\alpha).$$
    Moreover, for $\alpha'\in (\alpha_1,\alpha_2)$, if $h(-\alpha')=\mathscr{P}(t')-t'\alpha'$ and $\mathscr{P}$ is strictly convex at $t=t'$, then $\mathscr{B}$ is strictly concave and $C^1$ at $\alpha=\alpha'$. 
\end{enumerate}
\end{proposition}

The proof of Proposition 9.1 can be found in $\mathsection 3.4$ in \cite{Cli10}. Here we give a few remarks on how to understand it. Proposition 9.1(1) is a general result which holds for all continuous map on compact space, without any requirements on the dynamics. Essentially it is a rewriting of variational principle. Proposition 9.1(3) gives the reverse implication. We know from the upper-semi continuity of pressure map and the uniqueness of equilibrium states $\mu_t$ for $\varphi_t$ with $t<1$ that the pressure function $\mathscr{P}(t)$ is $C^1$ when $t<1$. In fact, for each $\alpha \in (\alpha_1,\alpha_2)$, there is a unique supporting line $l_{\alpha}$ tangent to $\mathscr{P}(t)$ with slope being $\alpha$. Observe that the $y$-coordinate of the intersection of $l_{\alpha}$ and $y$-axis is just $\mathscr{E}(\alpha)$, the Legendre transform of $\mathscr{P}$ at $\alpha$.  Meanwhile, the slope of the tangent line of $\mathscr{P}$ at each $t<1$ is also known to be $\int{\varphi^{geo} d_{\mu_t}}$ by classic results (see for example Proposition 3.4.3 in \cite{Cli10}). Therefore, if $l_{\alpha}$ intersects the graph of $\mathscr{P}$ at $(t_{\alpha},\mathscr{P}(t_{\alpha}))$, we have
\begin{equation}
\mathscr{P}'(t_\alpha)=\int{\varphi^{geo} d{\mu_{t_{\alpha}}}}=\alpha.
\end{equation}
By $\mathscr{P}(t_{\alpha})=h(\mu_{t_{\alpha}})+t_{\alpha}\int{\varphi^{geo} d{\mu_{t_{\alpha}}}}$, we immediately see that $h(\mu_{t_{\alpha}})=\mathscr{E}(\alpha)$. The definition of $\mu_{t_{\alpha}}$ shows that it is ergodic, thus supported on $L(-\alpha)$ by (9.1). Therefore, $h(-\alpha)\geq \mathscr{E}(\alpha)$ by variational principle. According to Proposition 3.1(1) and the classic properties of Legendre transform of convex function, we know $h(-\alpha)\leq  \mathscr{E}(\alpha)$. In conclusion, we have $h(-\alpha)=\mathscr{E}(\alpha)$ for $\alpha\in (\alpha_1,\alpha_2)$. 

\subsection{Proof of $L(\alpha)$ being non-empty for $\alpha\in (\alpha_1,0)$}

In the last section, we just defined $l_{\alpha}$ to be the supporting line of $\mathscr{P}(t)$ with slope $\alpha$ for all $\alpha\in (\alpha_1,\alpha_2)$. In fact, we can extend this definition of to all $\alpha\in (\alpha_1,0]$ and each $l_{\alpha}$ intersects $y$-axis at $(0,\mathscr{E}(\alpha))$.

Due to the existence of neutral fixed point at the origin, we know $\mathscr{P}(t)=0$ for all $t\geq 1$. As a result, $l_{\alpha}(t)$ intersects the $x$-axis at $t=1$ when $\alpha \in [\alpha_2,0]$. For those $\alpha$, $\mathscr{E}(\alpha)=-\alpha$. Then a natural question is to ask if $h(-\alpha)=\mathscr{E}(\alpha)=-\alpha$ for $\alpha \in [\alpha_2,0]$. To answer this question, we will need $L(-\alpha)$ to be at least non-empty for those greater $\alpha$, which is the main proposition that we will prove in this section.

\begin{proposition}
$L(-\alpha)$ is non-empty for all $\alpha\in (\alpha_1,0]$.
\end{proposition}

We know the case for $\alpha\in (\alpha_1,\alpha_2)$ according to the analysis in $\mathsection 9.2$. We also know that the origin belongs to $L(0)$ by Proposition 3.1(3). As a result, we only need to verify that $L(-\alpha)$ is non-empty for $\alpha\in [\alpha_2,0)$. As introduced in $\mathsection 9.1$, we will construct an increasingly nested sequence of basic sets $\{\widetilde{\Lambda_i}\}_{i\in \mathbb{N}}$, $\widetilde{\Lambda}_i\subset \widetilde{\Lambda}_{i+1}$ and show that for any such $\alpha$, when $i$ is large enough, $\widetilde{\Lambda}_{i}$ will have non-trivial intersection with $L(-\alpha)$. In fact, this result follows from the following proposition:

\begin{proposition}
For any basic set $\Lambda\subset \mathbb{T}^2$, write $\mathscr{P}_{\Lambda}(t)$ as the topological pressure of $\varphi_t$ over $\Lambda$. There is an increasing sequence of basic sets $\{\widetilde{\Lambda_i}\}_{i\in \mathbb{N}}$, $\widetilde{\Lambda}_i\subset \widetilde{\Lambda}_{i+1}$ such that $\mathscr{P}_{\widetilde{\Lambda}_i}(t)\nearrow \mathscr{P}(t)$ pointwisely.
\end{proposition}

Let us show how Proposition 9.2 follows from this. We know that $\mathscr{P}$ is convex, $C^1$ for $t<1$ and $\mathscr{P}'(t)\searrow \alpha_1$ when $t\rightarrow -\infty$ by (9.1). We also know that $\mathscr{P}(t)=0$ for all $t \geq 1$. Now given any $\alpha\in [\alpha_2,0)$, we look at $t=2$. By Proposition 9.3, there exists $N_1\in \mathbb{N}$ such that when $n>N_1$, $\mathscr{P}_{\widetilde{\Lambda}_n}(2)>\alpha$. Meanwhile, since $\alpha\leq \alpha_2<\alpha_1$, there is some $T_{\alpha}<0$ sufficiently small such that $\mathscr{P}(T_{\alpha})>\frac{(\alpha_1+\alpha_2)(T_\alpha+1)}{2}$. Respectively we have some $N_2\in \mathbb{N}$ such that for all $n>N_2$, $\mathscr{P}_{\widetilde{\Lambda}_{n}}(T_{\alpha})>-\alpha(1+T_{\alpha})$. Finally, we know that $\widetilde{\Lambda}_n(t)$ is $C^1$ at all $t\in \mathbb{R}$ and $\mathscr{P}_{\widetilde{\Lambda}_{n}}(1)\leq \mathscr{P}(1)=0$ for all $n\in \mathbb{N}$. Then by the mean value theorem, we know for $n>\max\{N_1,N_2\}$, there is a point $T_{n,\alpha}$ such that $\mathscr{P}'_{\widetilde{\Lambda}_n}(T_{n,\alpha})=\alpha$. Now we use the classic result of uniqueness of equilibrium state for $\varphi_t$ over basic set and apply (9.1). This concludes that $L(\alpha)\cap \widetilde{\Lambda_{n}}$ is non-trivial for $n$ sufficiently large, therefore Proposition 9.2 follows from Proposition 9.3. Now let us prove Proposition 9.3. To satisfy the convergence in the pressure function, it suffices to  construct $\widetilde{\Lambda}_{n}$ in a way such that any basic set $\Lambda \subset \mathbb{T}^2$ is contained in $\widetilde{\Lambda}_m$ for some $m \geq 1$. This is because, by Katok horseshoe theorem \cite{KH95}, for any small $\epsilon_0>0$, hyperbolic ergodic $\mu$ and continuous $\varphi$, there is a basic set $\Lambda$ such that $P_{\Lambda}(\varphi)>P_{\mu}(\varphi)-\epsilon_0$. If such $\widetilde{\Lambda}_m$ exists, when $t<1$, we have $\mathscr{P}_{\widetilde{\Lambda}_m}(t)\geq P_{\Lambda}(\varphi_t)>P_{\mu_t}(\varphi_t)-\epsilon_0=\mathscr{P}(t)-\epsilon_0$. When $t\geq 1$, we know from $\mathsection 8.3$ that we can construct a hyperbolic periodic point $p'$ (which is not origin) whose Lyapunov exponent is smaller than $\epsilon_0$. Write $\mu_p$ to be the ergodic measure supported on the orbit of $p'$ and apply the Katok horseshoe theorem, we get a basic set $\Lambda^p$ such that $\mathscr{P}_{\Lambda^p}(t)>P_{\mu_p}-\epsilon_0>-2\epsilon_0$. Therefore, by covering $\Lambda^p$ using some $\widetilde{\Lambda}_{l}$, we have $\mathscr{P}_{\widetilde{\Lambda}_{l}}>-2\epsilon_0$. Since $\epsilon_0$ can be arbitrarily small, we know Proposition 9.3 follows from the following proposition.

\begin{proposition}
There is an increasing sequence of basic sets $\{\widetilde{\Lambda_i}\}_{i\in \mathbb{N}}$, $\widetilde{\Lambda}_i\subset \widetilde{\Lambda}_{i+1}$ such that for any basic set $\Lambda \subset \mathbb{T}^2$, there is some $n\geq 1$ such that $\Lambda\subset \widetilde{\Lambda}_n$.
\end{proposition}

We point out that for general cases, this result is not always true (see \cite{F06}).

Now we follow the construction in \cite{BG14} and prove Proposition 9.4. Essentially we first cover the $\widetilde{G}$-invariant hyperbolic sets using compact locally maximal $\widetilde{G}$-invariant hyperbolic set, then glue the transitive components together via a gluing process. The essential lemma in the covering of hyperbolic sets is an analog to Proposition 4.3 in \cite{BG14}.

\begin{lemma}
Given any compact $\widetilde{G}$-invariant hyperbolic $\Lambda\subset \mathbb{T}^2$ and any neighborhood $U$ of $\Lambda$, there is a compact locally maximal $\widetilde{G}$-invariant hyperbolic set $\Lambda'$ such that $\Lambda\subset \Lambda' \subset U$.
\end{lemma}

\begin{proof}
First we make the observation that any compact $\widetilde{G}$-invariant hyperbolic set $\Lambda^*$ has topological dimension zero. This is because, the points on the same stable (unstable) leaf have the same forward (resp. backward) Lyapunov exponents. Therefore points on $W^s(\underaccent{\bar}{0})\cup W^u(\underaccent{\bar}{0})$ are disjoint from $\Lambda^*$. Moreover, since $W^s$ and $W^u$ are both dense and lie in small stable (unstable) cones, any two points in $\Lambda^*$ could be isolated using four boundaries of a su-rectangle, which is defined to be a closed rectangle formed by intersecting two pairs of segments in stable and unstable leaves of $W^s(\underaccent{\bar}{0})$ and $W^u(\underaccent{\bar}{0})$. In particular, this shows that $\Lambda^*$ is totally disconnected, thus has topological dimension equal to $0$.

Since $\Lambda$ is hyperbolic, by structural stability we might assume $U$ to be small enough so that any $\widetilde{G}$-invariant set contained in $U$ is also hyperbolic. The idea is to construct $\Lambda'$ as the image of a subshift of finite type under a continuous and injective map, thus inherits the natural local product structure property from shift space and is $\widetilde{G}$-invariant and compact.

We first construct $S\subset \mathbb{T}^2$ as a finite collection of disjoint closed su-rectangles. We require the union of rectangles to intersect with any orbit segments of fixed length $l\in \mathbb{N}$. This is always possible as we could simply cover the perturbed neighborhood of the origin using one su-rectangle and the rest follows from $\widetilde{G}$ being a linear toral automorphism there. For any $x\in S$, $\tau(x)\geq 1$ is defined to be the first return time to $S$ and the return map is defined to be $\mathcal{F}(x):=\widetilde{G}^{\tau(x)}(x)$. We will construct a shift space using $\mathcal{F}$.

Together with $\mathcal{F}$ we will also construct a sequence of subsets of $S$ on which $\mathcal{F}$ will be acting on. We claim there is a collection of su-rectangles $\mathscr{K}$ such that:

\begin{enumerate}
    \item The collection $\mathscr{K}$ is finite;
    \item The sets in $\mathscr{K}$ are compact and mutually disjoint;
    \item Each set in $\mathscr{K}$ is contained in a single su-rectangle of $S$ and is itself an su-rectangle with sufficiently small diameter (at a scale the shadowing process is feasible and the shadowing orbit is contained in $U$, see the later discussion for details);
    \item $\mathcal{F}$ is smooth on each set in $\mathscr{K}$;
    \item Each $K\in \mathscr{K}$ contains at least one point of $\Lambda\cap S$ and $\Lambda\cap S$ is contained in $\bigcup_{K\in \mathscr{K}}\text{Int}(K)$.
\end{enumerate}

Let's briefly explain on why it is possible to construct $\mathscr{K}$ satisfying all these properties. We have seen that $\Lambda$ is disjoint from the stable and unstable leaves of origin. As $\Lambda$ is closed, for any closed su-rectangle $\mathscr{R}$ that covers $\Lambda \cap S$ we can further find out a collection of smaller closed su-rectangles which is contained in $\mathscr{R}$ and still contains $\Lambda \cap S$ by cutting through the stable and unstable leaves of origin and shrink the boundary of the smaller su-rectangles. By repeating this process, we are able to get a collection of disjoint closed su-rectangles $\mathscr{R}_n$ such that $\Lambda \cap S\subset \bigcup_{R_n\in \mathscr{R}_n}R_n$ and each $R_n\in \mathscr{R}_n$ is a subset of some $R_{n-1}\in \mathscr{R}_{n-1}$. We can make the diameter of su-rectangles in $\mathscr{R}_n$ to be arbitrarily small as $W^s(\underaccent{\bar}{0})$ and $W^u(\underaccent{\bar}{0})$ are dense. We also know from above that the boundary of any su-rectangle does not intersect $\Lambda$, the union of interior of all such $R_n$ contains $\Lambda \cap S$, forming an open cover. The $\mathscr{K}$ is thus defined to be the minimized finite open cover derived from $\text{Int}(R_n)$ where $R_n\subset \mathscr{R}_n$ and $n$ is chosen to be sufficiently large such that the diameter of $R_n$ is sufficiently small.

We are using elements from $\mathscr{K}$ to make the symbols in the target shift space. Following \cite{BG14} we define $\mathscr{K}_N$ to be the collection of sets with the form
$\bigcap_{j=-N}^{j=N}\mathcal{F}^{-i}K_j,$
where $K_j\in \mathscr{K}$.

Using elements in $\mathscr{K}_N$ will allow us to construct a natural shift space. As in \cite{BG14}, A bi-infinite sequence $\{K_i^N\}_{i=-\infty}^{\infty}$ in $\mathscr{K}_N$ is called $N$-admissible if for any $i\in \mathbb{Z}$, there is some $x_i\in K_i^N\cap \Lambda$ such that $\mathcal{F}(x_i)\in K_{i+1}^N\cap \Lambda$. For each $N$ we notice that the set of all $N$-admissible sequences is a subshift of finite type as all forbidden blocks are of length two and the choice on symbols is finite. It is also proved in \cite{BG14}, Lemma 4.4 that the diameter of $\mathscr{K}_N$ decreases to $0$ as $N\rightarrow \infty$. Therefore, if $N$ is made large enough, the sequence of points $\{x_i^N\}_{i=-\infty}^{\infty}$ with $x_i^N\in K_i^N$ for all $i\in \mathbb{Z}$, will become a pseudo-orbit that can be shadowed by an unique orbit that stays close to $\Lambda\cap S$ all the time. Here the shadowing property of $\widetilde{G}$ (thus $\mathcal{F}$) at all scales comes from the homeomorphic conjugacy of $\widetilde{G}$ to linear toral automorphism $f_A$, which is Anosov and has shadowing property at all scales. This shadowing only depends on the sequence and is independent of the choice of $\{x_i^N\}$ by expansiveness of $\widetilde{G}$. Therefore, the shadowing map $\Phi^N$ from $N$-admissible sequence to the shadowing point is well-defined and obviously continuous. Write $\Lambda^N$ to be the image of $\Phi^N$. By making N large and the diameter of elements in $\mathscr{K}$ sufficiently small, we can have $\Phi^N$ to be injective and $\Lambda^N$ to be contained in $U$. 

Finally let's see how $\Lambda^N_{*}:=\bigcup_{i=0}^{l}\widetilde{G}^i(\Lambda^N)$, which is the union of all the orbits passing through $\Lambda^N$, gives the required basic set. First of all we have the $\mathcal{F}$-orbit of $\Lambda\cap S$ to be contained in $\Lambda^N$ as the real orbit always shadows itself. Since $\bigcup_{j=1}^{l}\widetilde{G}^j(S)=\mathbb{T}^2$, we have $\Lambda\subset \Lambda^N_{*}$. For the same reason we could make diameter of $\mathscr{K}$ to be small enough so that $\widetilde{G}^i(\Lambda^N)\subset U$ for all $0\leq i \leq l$, which makes $\Lambda^N_{*}\subset U$. As $\Lambda^N$ is the image of a compact set under continuous map $\Phi^N$, it is compact, which makes $\Lambda^N_{*}$ compact. $\Lambda^N_{*}$ is obviously $\widetilde{G}$-invariant and is hyperbolic by the earlier restriction on $U$. Finally, Since $\Phi^N$ is bijective and continuous, $\Lambda^N$ inherits the natural local product structure from shift space, so does $\Lambda^N_{*}$. In conclusion, all the required results are satisfied and $\Lambda^N_{*}$ is the desired $\Lambda'$ we are looking for, given N large enough and diameter of $\mathscr{K}$ small enough.
\end{proof}
Now we state the second essential lemma about gluing the basic sets together. This is analogous to Proposition 5.3 in \cite{BG14}.

\begin{lemma}
Given two basic sets $\Lambda'$ and $\Lambda''$ in $\mathbb{T}^2$, there exists a basic set $\Lambda'''\subset \mathbb{T}^2$ such that $\Lambda' \cup \Lambda'' \subset \Lambda'''$.
\end{lemma}
\begin{proof}
For any $x\in \Lambda'$ and any open neighborhood $U$ of $x$, $\{\widetilde{G}^i(U)\}_{i\in \mathbb{Z}}$ is an open cover of $\Lambda'$ as $\widetilde{G}$ acts transitively. By compactness of $\Lambda'$ and the fact that $\widetilde{G}^i(U) \cap \Lambda'$ contains $\widetilde{G}^i(x)$ for each $i\in \mathbb{Z}$, thus being non-empty, we know there is some $i\in \mathbb{Z}$ such that $U\cap \widetilde{G}^i(U) \cap \Lambda'\neq \emptyset$. This indicates that $\Omega(\widetilde{G}\vert_{\Lambda'} )=\Lambda'$, and for $\Lambda''$ is similar. Since $\widetilde{G}$ is conjugate to a linear automorphism via homeomorphism, $\Omega(\widetilde{G})=\mathbb{T}^2$. As $\widetilde{G}$ acts transitively over $\mathbb{T}^2$, $\Lambda'$ and $\Lambda''$, from Theorem 5.10 in \cite{Wal82} we know the action is also one-sided transitive. Together with local product structure over $\mathbb{T}^2$, $\Lambda'$ and $\Lambda''$, we can find $v'\in \Lambda'$, $v''\in \Lambda''$ and $w\in \mathbb{T}^2$ such that the forward and backward orbits of $v'$, $v''$ and $w$ are both dense in $\Lambda'$, $\Lambda''$ and $\mathbb{T}^2$.

We claim that we can find $w''\in \mathbb{T}^2$ such that the orbit of $w'$ is forward asymptotic to the orbit of $v''$ and backward asymptotic to the orbit of $v'$. As the orbit of $w$ is backward dense in $\mathbb{T}^2$, we can find some $i>0$ such that $d(\widetilde{G}^{-i}(w),v')<\epsilon$, where $\epsilon$ is as in the first seven sections. We can thus use a local product structure of $\widetilde{G}$ at scale $\epsilon$ to find $w'$, which is backward asymptotic to orbit of $v'$ and forward asymptotic to orbit of $w$. Since $w$ is forward dense, there is some $j>0$ such that $d(\widetilde{G}^j(w),v'')<\epsilon$, so is $\widetilde{G}^{i+j}(w')$. We then use local product structure between $\widetilde{G}^{i+j}(w')$ and $v''$ to get $w''$, which is forward asymptotic to $v''$ and backward asymptotic to the orbit of $\widetilde{G}^{i+j}(w')$, thus backward asymptotic to the orbit of $v'$. 

Similarly we can find $w''\in \mathbb{T}^2$ which is forward asymptotic to the orbit of $v'$ and backward asymptotic to the orbit of $v''$. Write $\Lambda$ as the union of orbit of $w''$, orbit of $w'$, $\Lambda'$ and $\Lambda''$. $\Lambda$ is compact as the orbit of $w''$ and $w'$ are forward and backward asymptotic to locally maximal sets. It is obviously $\widetilde{G}$-invariant and hyperbolic. We can then cover it using a compact $\widetilde{G}$-invariant locally maximal hyperbolic set $\widetilde{\Lambda}$. We notice that $\widetilde{G}$ acts on $\Lambda$ transitively and $\Lambda\subset \Omega(\widetilde{G}\vert_{\widetilde{\Lambda}})$ by shadowing orbit of $w'$ and $w''$ using a loop between $\Lambda'$ and $\Lambda''$. The shadowing orbit is in $\widetilde{\Lambda}$ as $\widetilde{\Lambda}$ is locally maximal. Therefore, $\Lambda$ must lie in one component of the spectral decomposition for $\Omega(\widetilde{G}\vert_{\widetilde{\Lambda}})$, which is the desired basic set covering $\Lambda'$ and $\Lambda''$.
\end{proof}
Using Lemma 9.5 and 9.6, we are able to prove lemma Proposition 9.4. We begin the construction of $\widetilde{\Lambda}_m$ with defining an increasingly nested sequence $\{\Lambda_n\}_{n\in \mathbb{N}}$ as $\Lambda_n=\overline{\text{Per}(\mathbb{T}^2\setminus B(\frac{1}{10n}))}$, where $\text{Per}(E)$ is the set of periodic orbits of $\widetilde{G}$ whose entire orbit lies in $E$ for $E\subset \mathbb{T}^2$ and $B(\delta)$ is the open ball centered at $\underaccent{\bar}{0}$ with radius $\delta>0$. It is obvious that $\{\Lambda_n\}$ forms an increasingly nested sequence of compact $\widetilde{G}$-invariant hyperbolic set. We know from Lemma 9.5 that there is some $\Lambda_n'$ containing $\Lambda_n$ which is locally maximal, compact, $\widetilde{G}$-invariant and hyperbolic. Since $\Omega(\widetilde{G}\vert_{\Lambda_n'})$ contains all the periodic points in $\Lambda_n'$ and is closed, it contains $\Lambda_n$. We can then use the spectral decomposition for $\Omega(\widetilde{G}\vert_{\Lambda_n'})$ and apply Lemma 9.6 multiple times to find a basic set $\widetilde{\Lambda}_n$ which eventually covers $\Lambda_n$. As $\{\Lambda_n\}$ is increasingly nested, we can also choose $\{\widetilde{\Lambda}_n\}$ to be increasingly nested by applying Lemma 9.6. From shadowing lemma for basic sets and transitivity of $\widetilde{G}$, we know for any basic set $\Lambda$ we have $\Lambda\subset \bigcup_{n\in \mathbb{N}}\widetilde{\Lambda}_n$. Since each $\widetilde{\Lambda}_n$ is locally maximal, there is a nested sequence of open sets $U_n$ such that $\widetilde{\Lambda}_n=\bigcap_{i\in \mathbb{Z}}\widetilde{G}^i(U_n)$. In particular $\widetilde{\Lambda}_n\subset U_n$ for all $n$. Therefore we have $\Lambda\subset \bigcup_{n\in \mathbb{N}}U_n$. Since $\Lambda$ is compact, there is some $m\in \mathbb{N}$ such that $\Lambda\subset U_m$. It follows from $\widetilde{G}$-invariance of $\Lambda$ that $\Lambda\subset \widetilde{\Lambda}_m$, which concludes Proposition 9.4, thus Proposition 9.3, which in turn proves Proposition 9.2.

\subsection{Proof of Theorem 1.4}

Finally we are at the stage of proving Theorem 1.4, which is our main theorem in the multifractal analysis of the Katok map $\widetilde{G}$. 

From Proposition 9.3 we know that for for any $\alpha\in (\alpha_1,0)$, there is some $N_{\alpha}\in \mathbb{N}$ such that $L(-\alpha)\cap \widetilde{\Lambda}_n\neq \emptyset$ for $n\geq N_{\alpha}$. Since $\{\widetilde{\Lambda}_n\}_{n\in \mathbb{N}}$ is increasingly nested, if we write $\dim_n(-\alpha)$ as the Hausdorff dimension of $L(-\alpha)\cap \widetilde{\Lambda}_n$ for all $n\geq N_{\alpha}$, then we immediately get $d_H(-\alpha)\geq \lim_{n\rightarrow \infty}\dim_n(-\alpha)$. Recall we use $l_{\alpha}$ to represent the supporting line to $\mathscr{P}$ with slope being $\alpha$ and $l_{\alpha}$ is well-defined for any $\alpha\in (\alpha_1,0]$. Similarly we use $l_{\alpha}^n$ to represent the supporting line to $\mathscr{P}_{\widetilde{\Lambda}_n}$. By (9.1) in the uniformly hyperbolic version, we have $l_{\alpha}^n$ to be well-defined for all $n\geq N_{\alpha}$ and $\alpha\in (\alpha_1,0)$, which is tangent to $\mathscr{P}_{\widetilde{\Lambda}_n}$ at the point $(t_{\alpha}^n,\mathscr{P}_{\widetilde{\Lambda}_n}(t_{\alpha}^n))$.

Notice that $\mathscr{P}'_{\widetilde{\Lambda}_n}(t_{\alpha}^n)=\int \varphi^{geo}d\mu_{\alpha}^n=\alpha$, where $\mu_{\alpha}^n$ is the unique equilibrium state of $\varphi_{t_{\alpha}^n}$ with $\widetilde{G}$ over $\widetilde{\Lambda}_n$. In particular, $\mu_{\alpha}^n$ is ergodic, $\chi(\mu_{\alpha}^n)=-\alpha$ and $h(\mu_{\alpha}^n)=\mathscr{E}_{\widetilde{\Lambda}_n}(\alpha)$, where $\mathscr{E}_{\widetilde{\Lambda}_i}(\alpha)=\inf_{t\in \mathbb{R}}(\mathscr{P}_{\widetilde{\Lambda}_i}(t)-t\alpha)$ is the Legendre transform of $\mathscr{P}_{\widetilde{\Lambda}_i}(t)$ at $\alpha$.  Since $\widetilde{G}$ is area-preserving, the Main Theorem in \cite{You82} tells us that $\dim_H(\mu_{\alpha}^n)=\frac{2h(\mu_{\alpha}^n)}{-\alpha}$, where $\dim_H(\mu_{\alpha}^n):=\inf\{\dim_H(E): \mu_{\alpha}^n(E)=1\}$. Since $\mu_{\alpha}^n$ is supported on $L(-\alpha)\cap \widetilde{\Lambda}_n$, we immediately have the following result:

\begin{lemma}
For any $\alpha\in (\alpha_1,0)$, there exists $N_{\alpha}\in \mathbb{N}$ such that for all $n\geq N_{\alpha}$, $dim_n(-\alpha)\geq \frac{2\mathscr{E}_{\widetilde{\Lambda}_n}(\alpha)}{-\alpha}$ and $d_H(-\alpha)\geq \lim_{n\rightarrow \infty}\frac{2\mathscr{E}_{\widetilde{\Lambda}_n}(\alpha)}{-\alpha}$ where $\dim_n(-\alpha)$ is the Hausdorff dimension of $L(-\alpha)\cap \widetilde{\Lambda}_n$ for all $n\geq N_{\alpha}$ and $d_H(-\alpha)$ is the one for $L(-\alpha)$.
\end{lemma}

From Lemma 9.7, we know that in order to get the lower bound of $d_H(-\alpha)$ in Theorem 1.4, it suffices to prove that $\lim_{n\rightarrow \infty}\mathscr{E}_{\widetilde{\Lambda}_n}(\alpha)=\mathscr{E}(\alpha)$ for all $\alpha \in (\alpha_1,0)$, which directly follows from the following lemma:

\begin{lemma}
For any $\alpha\in (\alpha_1,0)$ and any $t\in \mathbb{R}$, $l_{\alpha}^n(t)$ increases to $l_{\alpha}(t)$. In other words, the supporting line of $\mathscr{P}_{\widetilde{\Lambda}_n}$ with slope $\alpha$ converges monotonically to the supporting line of $\mathscr{P}$ with the same slope. In particular, $\lim_{n\rightarrow \infty}\mathscr{E}_{\widetilde{\Lambda}_n}(\alpha)=\mathscr{E}(\alpha)$.
\end{lemma}

\begin{proof}
Fix any $\alpha\in (\alpha_1,0)$, we have $N_{\alpha}$ as before. From convexity of $\mathscr{P}_{\widetilde{\Lambda}_n}(t)$, when $n\geq N_{\alpha}$, There is some closed interval $[a_n,b_n]$ such that $\mathscr{P}_{\widetilde{\Lambda}_n}(t)>l_{\alpha}(t)$ for t not in this interval. Meanwhile, by $\mathscr{P}_{\widetilde{\Lambda}_j}(t)\nearrow \mathscr{P}(t)$ as $j\rightarrow \infty$ and the continuity of both $\mathscr{P}_{\widetilde{\Lambda}_j}(t)$ and $\mathscr{P}(t)$, the convergence over $[a_n,b_n]$ is uniform by Dini's Theorem. Therefore, for any small $\delta>0$, there is some $N=N(\alpha,n,\delta)\geq n$ such that $\mathscr{P}_{\widetilde{\Lambda}_j}(t)\geq \mathscr{P}(t)-\delta\geq l_{\alpha}(t)-\delta$ for all $t\in [a_n, b_n]$ and $j\geq N$, which in return shows that $\mathscr{P}_{\widetilde{\Lambda}_j}(t)\geq l_{\alpha}(t)-\delta$ for $j\geq N$ and all $t$. In particular, it follows that $l_{\alpha}^j \geq l_{\alpha}-\delta$ for all $j\geq N$, which leads to the convergence from $\mathscr{E}_{\widetilde{\Lambda}_n}(\alpha)$ to $\mathscr{E}(\alpha)$ for all $\alpha\in (\alpha_1,0)$.
\end{proof}

In conclusion, the Hausdorff dimension estimate in Theorem 1.4 follows from Lemma 9.7 and 9.8.

\begin{proposition}
For any $\alpha\in (\alpha_1,0)$, we have $d_H(-\alpha)\geq \frac{2\mathscr{E(\alpha)}}{-\alpha}$. In particular, when $\alpha \in [\alpha_2,0)$, we have $d_H(-\alpha)=2$.
\end{proposition}

It remains to show that $h(-\alpha)=\mathscr{E}(\alpha)$ for $\alpha\in (\alpha_1,0]$. By Proposition 9.1, it suffices to show the case where $\alpha\in [\alpha_2,0]$. 

Define $\chi(\mu):=\int{\varphi^{geo}d\mu}$, which is the average of Lyapunov exponents for every ergodic component in the decomposition of $\mu$. We notice that this definition does not cause ambiguity when $\mu$ is itself ergodic.

\begin{lemma}
For any $\alpha\in [\alpha_2,0]$, there is some $\mu_{\alpha}\in \mathcal{M}(\widetilde{G})$ such that $l_{\alpha}=h(\mu_{\alpha})-t\int{\varphi^{geo}d{\mu_{\alpha}}}$. In particular, $\chi(\mu_{\alpha})=-\alpha$ and $\mathscr{E}(\alpha)=h(\mu_{\alpha})$.
\end{lemma}

\begin{proof}
We first show the above lemma holds for $\alpha=\alpha_2$. For $t\nearrow 1$, $\mathscr{P}(t)$ is $C^1$ and $\mathscr{P}'(t)=-\int{\varphi^{geo}}d\mu_t$. Let $\mu'$ be a weak*-limit of $\mu_t$. Since $\varphi^{geo}$ is continuous and $\widetilde{G}$ is expansive, we have $\mathscr{P}(1)=\limsup_{t\nearrow 1}\mathscr{P}(t)=\limsup_{t\nearrow 1}(h(\mu_t)-t\int{\varphi^{geo}d{\mu_t}})\leq h(\mu')-\int{\varphi^{geo}d{\mu'}}$. By variational principle, we have $0=\mathscr{P}(1)=h(\mu')-\int{\varphi^{geo}d{\mu'}}$. Meanwhile, $\chi(\mu')=\int{\varphi^{geo}d{\mu'}}=\lim_{t\nearrow 1}{\int{\varphi^{geo}d{\mu_t}}}=\lim_{t \nearrow 1}\chi(\mu_t)=-\alpha_2$. Therefore, $l_{\alpha_2}=h(\mu')-t\int{\varphi^{geo}d{\mu'}}$ and $\mu'=\mu_{\alpha_2}$. 

For $\alpha\in (\alpha_2,0]$, a suitable linear combination of $\mu_{\alpha_2}$ and $\delta_{\underaccent{\bar}{0}}$ will give us $\mu_{\alpha}\in \mathcal{M}(\widetilde{G})$ such that $\chi(\mu_{\alpha})=-\alpha$ and $\mathscr{P}_{\mu_{\alpha}}(1)=0$ as the entropy map is affine in measure. This shows that $l_\alpha=h(\mu_{\alpha})-t\int{\varphi^{geo}d\mu_{\alpha}}$ is the supporting line for $\mathscr{P}(t)$ with slope $\alpha$ for $\alpha \in (\alpha_2,0]$. Lemma 9.10 now just comes from a combination of the results in the two parts above.
\end{proof}

From Proposition 9.1, Lemma 9.10 and variational principle, we have 

\begin{lemma}
$\mathscr{E}(\alpha)=\max\{h(\mu): \mu\in \mathcal{M}(G), \chi(\mu)=-\alpha\}$ for all $\alpha \in (\alpha_1,0]$
\end{lemma}

Finally, since $L(-\alpha)$ is non-empty for all $\alpha\in (\alpha_1,0]$ and $\widetilde{G}$ has specification property, we apply Theorem 3.5 in \cite{TV03} and conclude that $h(-\alpha)=\mathscr{E}(\alpha)$ for all $\alpha\in (\alpha_1,0]$.

\begin{proposition}
For any $\alpha\in (\alpha_1,0]$, $h(-\alpha)=\mathscr{E}(\alpha)$.
\end{proposition}

Combining the results from Proposition 9.9 and Proposition 9.12, we conclude the proof of Theorem 1.4.

\subsection*{Acknowledgments} I would like to thank my advisor, Dan Thompson, for providing enlightening suggestions through the process. I also thank Ke Zhang for some enlightening conversations and the referee for the helpful comments that have benefited the paper a lot.

%\subsection*{Acknowledgments} 

%\bibliographystyle{amsplain}
%\bibliography{}

\end{document}